\documentclass[a4paper,11 pt]{article}

\usepackage[T1]{fontenc}
\usepackage[english]{babel}
\usepackage{times}

\usepackage{amsmath}
\usepackage{amsfonts}
\usepackage{amssymb}
\usepackage{amsmath}
\usepackage{graphicx}
\usepackage{theorem}
\usepackage{array}

\usepackage{color}

\usepackage{mathrsfs}

\evensidemargin\oddsidemargin

\newtheorem{theorem}{Theorem}
\newtheorem{lemma}{Lemma}[section]
\newtheorem{proposition}[lemma]{Proposition}

{\theorembodyfont{\rmfamily}\newtheorem{definition}{Definition}[section]}

{\theorembodyfont{\rmfamily}\newtheorem{remark}{Remark}[section]}
{\theorembodyfont{\rmfamily}\newtheorem{comment}{Comment}[section]}

{\theorembodyfont{\rmfamily}}

{\theorembodyfont{\rmfamily}}

{\theorembodyfont{\rmfamily}}

\newenvironment{proof}[1][Proof]{\begin{trivlist}
\item[\hskip \labelsep {\bfseries #1}]}{\end{trivlist}}

\newcommand{\qed}{\hfill $\square$}

\def\un{\mathbf{1}} 

\def\cqfd

\def\ee{\varepsilon}

\newcommand{\noi}{\noindent}

\newcommand{\Ppsi}{\Psi}

\newcommand{\ee}{\varepsilon}

\newcommand{\ree}{\mathtt{e}}

\newcommand{\rN}{\mathrm{N}}

\def\cq{ \hfill $\square$}
\def\cqfd{ \hfill $\blacksquare$}

\def\rZ{\mathtt{Z}}

\newcommand{\bE}{\mathbf{E}}

\newcommand{\bN}{\mathbf{N}}

\newcommand{\bP}{\mathbf{P}}

\newcommand{\bbD}{\mathbb{D}}
\newcommand{\bbE}{\mathbb{E}}

\newcommand{\bbN}{\mathbb{N}}

\newcommand{\bbP}{\mathbb{P}}
\newcommand{\bbQ}{\mathbb{Q}}
\newcommand{\bbR}{\mathbb{R}}

\newcommand{\cE}{\mathcal{E}}
\newcommand{\cF}{\mathcal{F}}

\newcommand{\cM}{\mathcal{M}}
\newcommand{\cN}{\mathcal{N}}

\newcommand{\cZ}{\mathcal{Z}}

\newcommand{\ccB}{\mathscr{B}}

\newcommand{\ccE}{\mathscr{E}}
\newcommand{\ccF}{\mathscr{F}}
\newcommand{\ccG}{\mathscr{G}}

\newcommand{\ccM}{\mathscr{M}}

\newcommand{\ccP}{\mathscr{P}}
\newcommand{\ccQ}{\mathscr{Q}}

\begin{document}

\title{ \textbf{\textsc{On the Eve property for CSBP}}}
\author{Thomas {\sc Duquesne}, Cyril {\sc Labb\'e} 
\thanks{
\textsl{Mail address:} LPMA,
Univ.~P.~et M.~Curie (Paris 6), Bo\^ite 188, 4 place Jussieu, 75252 Paris Cedex 05, FRANCE.
\textsl{Email:} thomas.duquesne@upmc.fr and cyril.labbe@upmc.fr} }
\vspace{2mm}

\date{\today}

\maketitle

\begin{abstract} We consider the population model associated to 
continuous state branching processes and we are interested in the so-called Eve property that asserts the existence of an ancestor with an overwhelming progeny at large times, and more generally, in the possible behaviours of the frequencies among the population at large times. 
In this paper, we classify all the possible behaviours according to the branching mechanism of the continuous state branching process.

\medskip

\end{abstract}

\section{Introduction}
\label{introsec}

Continuous State Branching Processes (CSBP for short) have been introduced by Jirina \cite{Jir58} 
and Lamperti \cite{Lam67b, Lam67a, Lam67c}. They are the scaling limits 
of Galton-Watson processes: see  
Grimvall \cite{Gri74} and Helland \cite{Hel78} for general functional limit theorems. 
They represent the random evolution of the size of a \textit{continuous} population. Namely, if $Z=(Z_t)_{t\in [0, \infty)}$ is a CSBP, the population at time $t$ can be 
represented as the interval $[0, Z_t]$. In this paper, we focus on the following question: \textit{as $t \! \rightarrow \! \infty$, does the population concentrate on the progeny of a single ancestor $\ree \! \in \! [0, Z_0]$ ?} If this holds true, then we say that \textit{the population has an Eve}. More generally, we 
discuss the asymptotic frequencies of settlers.
A more formal definition is given further in the introduction. 

 The Eve terminology was first introduced by Bertoin and Le Gall~\cite{BerLeG03} for the generalised Fleming-Viot process. Tribe~\cite{Tribe92} addressed a very similar question for super-Brownian motion with quadratic branching mechanism, while in Theorem 6.1~\cite{DK99} Donnelly and Kurtz gave a particle system interpretation of the Eve property. In the CSBP setting, the question has been raised for a general branching mechanism in~\cite{Labb�13}. Let us mention that Grey~\cite{Gre74} and Bingham~\cite{Bin76} introduced martingale techniques to study the asymptotic behaviours of CSBP under certain assumptions on the branching mechanism: to answer the above question in specific cases, we extend their results using slightly different tools. For related issues, we also refer to Bertoin, Fontbona and Martinez~\cite{BerFonMar08}, Bertoin~\cite{Ber09} and Abraham and Delmas~\cite{AbraDel09}. 
 
 Before stating the main result of the present paper, we briefly recall basic properties of CSBP, whose proofs can be found in Silverstein \cite{Sil68}, Bingham \cite{Bin76}, Le Gall \cite{LeG99} or 
Kyprianou \cite{Kyp06}. CSBP are $[0, \infty]$-Feller processes whose only two absorbing states are $0$ and $\infty$ and 
whose transition kernels $(p_t(x, \, \cdot \, ) \, ; \, t \! \in \! [0, \infty), x \! \in \! [0, \infty])$ satisfy the 
so-called \textit{branching property}:    
\begin{equation}
\label{branchprop}
 \forall x, x^\prime \in [0, \infty]\, , \; \forall t \in [0, \infty) \, , \quad  p_t (x, \, \cdot \, ) \ast 
p_t (x^\prime , \, \cdot \, )= p_t (x+x^\prime , \, \cdot) \; .
\end{equation}
Here, $\ast$ stands for the convolution product of measures. We do not consider CSBP that jump to $\infty$ on a single jump. Since the two absorbing points $0$ and $\infty$ belong to the state-space, the transition kernels are true probability measures 
on $[0, \infty]$ and they are characterised by their \textit{branching mechanism} 
$\Ppsi : [0, \infty) \rightarrow \bbR$ as follows: for any $t, \lambda , x \! \in \! [0, \infty )$, 
\begin{equation}
\label{udef}
\int_{[0, \infty)} \!\!\!   p_t (x, dy)  \,
\exp (- \lambda y )    
= 
\exp \! \big(\!-  x u(t, \lambda) \big) \; ,
\end{equation} 
where $u (\, \cdot\, , \lambda)$ is a 
$[0, \infty)$-valued function that satisfies $\partial_t u \, (t, \lambda) \! = \!  -\Ppsi (u(t, \lambda))$ and $u(0, \lambda) \! =\!  \lambda$. For short, we write CSBP($\Ppsi, x$) for continuous state 
branching process with branching mechanism $\Ppsi$ and initial value $x$. 
The branching mechanism $\Ppsi$ is necessarily of the following L\'evy-Khintchine form: 
\begin{equation}
\label{LKform}
\forall \lambda \in [0, \infty) \, , \quad 
\Ppsi(\lambda) = \alpha \lambda + \beta \lambda^2 + \int_{(0,\infty)} \!\!\! \!  \!\!\! \!   \pi (dr ) \, \big(e^{-\lambda r}-1+\lambda r  \un_{ \{r<1\}}\big)\, , 
\end{equation}
where $\alpha \! \in \! \bbR$, $\beta \! \geq \! 0$ and $\pi$ is a Borel measure on $(0,\infty)$ such that $\int_{(0,\infty)}(1\wedge r^2)\, \pi(dr)\! <\! \infty$. 
We recall that a CSBP with branching mechanism $\Ppsi$ 
is a time-changed spectrally positive L\'evy process whose Laplace exponent is $\Ppsi$: see for instance Lamperti \cite{Lam67b} and Caballero, Lambert and Uribe Bravo~\cite{CabLamUri09}. 
Consequently, the sample paths of a cadlag CSBP have no negative jump. Moreover, a CSBP has infinite variation sample paths iff the corresponding L\'evy process has infinite variation sample paths, which is equivalent to the following assumption:  
\begin{equation}
\label{infvar}
\hspace{-51mm}\textrm{(\textit{Infinite variation}) \hspace{18mm}}  \beta >0 \qquad \textrm{or} \qquad \int_{(0, 1)} \! \!\! r \, \pi (dr) = \infty \; . 
\end{equation}
Therefore, the finite variation cases correspond to the following assumption:
\begin{equation}
\label{fvar}
\hspace{-51mm}\textrm{(\textit{Finite variation}) \hspace{18mm}}  \beta =0 \qquad \textrm{and} 
\qquad \int_{(0, 1)} \! \!\! r \, \pi (dr) < \infty \; . 
\end{equation}
In the finite variation cases, $\Ppsi$ can be rewritten as follows: 
\begin{equation}\label{psifvar} 
\forall \lambda \in [0, \infty) , \quad \Ppsi (\lambda) = D \lambda -\!\! \int_{(0,\infty)} \!\!\! \! \!\!\! \!   \pi (dr ) \, (1-e^{-\lambda r} )\; ,  \quad \textrm{where} \quad D:= \alpha + \int_{(0,1)}  \!\!\! \! \!\!\! \! r \,  \pi (dr ) \; .
 \end{equation}
In these cases, note that $D= \lim_{\lambda \rightarrow \infty} \Ppsi (\lambda)/ \lambda$.

  We shall always avoid the cases of deterministic CSBP that correspond to linear branching mechanisms. Namely, \textit{we shall always assume that either $\beta >0$ or $\pi \neq 0$}.

  Since $\Ppsi$ is convex, it has a right derivative at $0$, that is possibly equal to $-\infty$. Furthermore, $\Ppsi$ has at most two roots. We introduce the following notation:  
\begin{equation}
\label{malthdef}
\Ppsi^\prime (0+):= \lim_{{\lambda \rightarrow 0+}} \lambda^{-1}\Ppsi (\lambda)\;  \in [-\infty, \infty) \quad 
\textrm{and} \quad  \gamma = \sup \big\{ \lambda \! \in \! [0, \infty ): \Ppsi (\lambda) \!\leq \! 0  \big\} \ . 
\end{equation}
Note that $\gamma \! > \! 0$ iff $\Ppsi^\prime (0+) \! <\!  0$, and that $\gamma \! =\! \infty$ iff $-\Ppsi$ is the Laplace exponent of a subordinator. 

We next discuss basic properties of the function $u$ defined by (\ref{udef}). 
The Markov property for CSBP entails 
\begin{equation}
\label{flow}
\forall t,s, \lambda \! \in \! [0, \infty) , \quad   u(t\!+\!s, \lambda)= u(t,u(s, \lambda)) \quad \textrm{\rm and} \quad \partial_t u \, (t, \lambda)\!=\! - \Ppsi (u(t, \lambda) )  , \, u(0, \lambda)\! = \!   \lambda  .
\end{equation} 
If $\lambda \!  \in \! (0, \infty)$, then 
$u(\, \cdot \, , \lambda)$ is the unique solution of (\ref{flow}). If 
$\lambda \! = \! \gamma $, then $u( \, \cdot \, , \gamma)$ is constant to $\gamma$. An easy argument derived  from (\ref{flow}) entails the following:  
if $\lambda \! >\! \gamma $ (resp.~$\lambda \! <\! \gamma $), then $u(\, \cdot \, , \lambda)$
is decreasing (resp.~increasing). 
Then, by an easy change of variable, (\ref{flow}) implies 
\begin{equation}
\label{integeq}
\forall t \in [0, \infty)  , \, \; \forall \lambda \in (0, \infty) \backslash \{ \gamma \}  , \qquad \int_{u(t, \lambda)}^{\lambda} \frac{du}{\Ppsi (u)} = t \; .
\end{equation}
 
For any $x\in [0, \infty]$, we denote by $\bbP_x$ the canonical law of CSBP($\Ppsi , x$) on the 
Skorohod space of cadlag $[0, \infty]$-valued functions that is denoted by 
$\bbD ([0, \infty), [0, \infty])$. We denote by $\rZ= (\rZ_{t})_{t\in [0, \infty)}$ the canonical process on $\bbD ([0, \infty), [0, \infty])$. As $t \! \rightarrow \! \infty$, a 
CSBP either converges to $\infty$ or to $0$. More precisely,  
\begin{equation} 
\label{asymZ}
\forall x \in (0, \infty)  , \quad  e^{-\gamma x}= \bbP_x \big( \lim_{^{t\rightarrow \infty}} \rZ_t = 0 \big)=1- \bbP_x \big( \lim_{^{t\rightarrow \infty}} \rZ_t = \infty \big) \; .
\end{equation}
If $\Ppsi^\prime (0+) \!  > \!  0$ (resp.~$\Ppsi^\prime (0+) \!  = \!  0$), then $\gamma \! = \! 0$ and the CSBP gets extinct: $\Ppsi$ is said to be \textit{sub-critical} (resp.~\textit{critical}). If $\Ppsi^\prime (0+) \! <\! 0$, then $\gamma \!>\! 0$ and the 
CSBP has a positive probability to tend to $\infty$: $\Ppsi$ is said to be \textit{super-critical}.  

Let us briefly discuss \textit{absorption}: let $\zeta_0$ and $\zeta_\infty $ be the times of absorption in resp.~$0$ and $\infty$. Namely:
\begin{equation}   
\label{absordef}    
\zeta_{0} \!=\!  \inf \big\{ t\! >\! 0: \textrm{$\rZ_t$ or $\rZ_{t-}= 0$}  \big\}  , \; \, 
\zeta_{\infty}\! =\! \inf \big\{ t\! >\! 0:  \textrm{$\rZ_t$ or $\rZ_{t-}= \infty$}  \big\}  \; \, {\rm and} \; \, \zeta \! =\!  \zeta_0 \wedge \zeta_\infty  , 
\end{equation} 
with the usual convention: $\inf \emptyset \! =\!  \infty$. We call $\zeta$ the time of absorption. The integral equation (\ref{integeq}) easily implies the following: 
\begin{equation}
\label{conserv}
\hspace{-20mm}\textrm{(\textit{Conservative} $\Ppsi$) \hspace{10mm}} 
\forall x \in [0, \infty) , \quad  \bbP_x ( \zeta_\infty <\infty) =0   \quad  \Longleftrightarrow \quad \int_{0+} 
\frac{\! dr}{(\Ppsi (r))_-} = \infty \; . 
\end{equation}
Here  $(\, \cdot \, )_-$ stands for the negative part function. If $\Ppsi$ 
is \textit{non-conservative}, namely if 
\begin{equation}
\label{nonconserv}
\hspace{-62mm}\textrm{(\textit{Non-conservative} $\Ppsi$) \hspace{30mm}} 
\int_{0+} 
\frac{\! dr}{(\Ppsi (r))_-} < \infty \; ,  
\end{equation}
then, $\Ppsi^\prime (0+)\! =\! -\infty$ and for any 
$t , x \! \in \! (0, \infty)$, $\bbP_x ( \zeta_\infty \!>\! t) \! =\!  \exp (\! -x\kappa (t))$, where $\kappa (t)\! :=\! \lim_{\lambda \rightarrow 0+} u(t, \lambda)$ satisfies 
$\int_{0}^{\kappa (t)} \! dr /(\Ppsi (r))_- \!=\! t$. 
Note that $\kappa : (0, \infty) \! \longrightarrow  \! (0, \gamma)$ is one-to-one and 
increasing. Thus, 
$\bbP_x $-a.s.~$\lim_{t\rightarrow \infty} \rZ_t \!=\! \infty $ iff $\zeta_\infty \! <\!  \infty $ and in this case, $\lim_{t\rightarrow \zeta_\infty -} \rZ_t \!=\! \infty$. Namely, the process reaches $\infty $ continuously.

The integral equation (\ref{integeq}) also implies the following: 
\begin{equation}
\label{persist}
\hspace{-19mm}\textrm{(\textit{Persistent} $\Ppsi$) \hspace{12mm}} 
\forall x \in [0, \infty) , \quad \bbP_x ( \zeta_0 <\infty) =0 \quad    \Longleftrightarrow \quad  
\! \int^{\infty} \frac{ dr}{\Ppsi (r)} = \infty \; . 
\end{equation}
If  \textit{$\Ppsi$ 
allows extinction in finite time}, namely if 
\begin{equation}
\label{nonpersist}
\hspace{-65mm}\textrm{(\textit{Non-persistent} $\Ppsi$) \hspace{30mm}} 
\int^{\infty} \frac{ dr}{\Ppsi (r)} < \infty \; ,  
\end{equation}
it necessarily implies that $\Ppsi$ satisfies (\ref{infvar}), 
namely that $\Ppsi$ is of infinite variation type.
In this case, for any 
$t , x \! \in\! (0, \infty)$, $\bbP_x ( \zeta_0 \! \leq \!  t) \! =\!  \exp (\! -x v(t))$ where $v (t) \! :=\! 
 \lim_{\lambda \rightarrow \infty} u(t, \lambda)$ satisfies 
$\int_{v(t)}^\infty \! dr/\Ppsi (r) \! =\!  t$. 
Note that $v : (0, \infty) \! \longrightarrow \! ( \gamma , \infty)$ is one-to-one and 
decreasing. Thus, $\bbP_x $-a.s.~$\lim_{t\rightarrow \infty} \rZ_t \! =\!  0$ iff $\zeta_0 \!<\!  \infty $. 

\smallskip

  The previous arguments allow to define $u$ for negative times. Namely, for all $t\! \in \! (0, \infty)$, set $\kappa (t)\!=\! \lim_{\lambda \rightarrow 0+} u(t, \lambda)$ and $v(t)\! =\!   \lim_{\lambda \rightarrow \infty} u(t, \lambda)$. As already mentioned, $\kappa (t)$ is positive if $\Ppsi$ is non-conservative and 
null otherwise and $v(t)$ is finite if $\Ppsi$ is non-persistent and infinite otherwise. Then, observe that 
$u(t, \, \cdot \, ) : (0, \infty) \! \longrightarrow \! (\kappa (t), v(t) )$ is increasing and one-to-one. We denote by $u(-t, \, \cdot \, ):  (\kappa (t), v(t) )\! \longrightarrow \! (0, \infty)$ the reciprocal function. 
It is plain that (\ref{integeq}) extends to negative times. Then, observe that $\partial_t u(-t, \lambda)= 
\Ppsi (u(-t, \lambda))$ and that (\ref{flow}) extends to negative times as soon as it makes sense.

\medskip

 Let us give here the precise definition of the Eve property. To that end, we fix $x \! \in \! (0, \infty)$ and denote by $\ccB([0, x])$ the Borel subsets of $[0, x]$. We also 
denote by $\ccM ([0, x])$ the set of positive Borel-measures on $[0, x]$ and by $\ccM_1([0, x])$ the set of Borel probability measures.   
Let us think of $m_t \! \in \! \ccM_1([0,x])$, $t \!\in \! [0, \infty)$, 
as the frequency distributions of a continuous population whose set of ancestors is $[0, x]$ and that evolves through time 
$t$. Namely for any Borel set $B\! \subset \!  [0, x]$, $m_t (B)$ is the frequency of the individuals at time $t$ whose ancestors belong to $B$. 
The relevant convergence mode is the 
\textit{total variation norm}:  
$$\forall  \mu, \nu \! \in \!  \ccM_1([0, x]) \, ,  \qquad \lVert \mu-\nu \rVert_{{\rm var}}= 
\sup \big\{ \, \lvert \mu (A) \!- \! \nu (A)\rvert \, ; \;  A\in \ccB([0, x])\,  \big\} \; .$$ 
Here, it is natural to assume that $t \mapsto m_t $ is cadlag in total variation norm. The Eve property can be defined as follows. 
\begin{definition}
\label{evedef} We denote by $\ell $ the Lebesgue measure on $\bbR$ (or its restriction to $[0, x]$ according to the context). 
Let $t\! \in \! (0, \infty)\! \longmapsto \! m_t\! \in\!  \ccM_1 ([0, x])$ be cadlag with respect to 
$\lVert \cdot  \rVert_{\textrm{var}}$ and assume that there exists $m_\infty \! \in \! \ccM_1 ([0, x])$ such that $\lim_{t\rightarrow \infty} \lVert m_t -m_\infty  \rVert_{\textrm{var}}= 0$, where 
\begin{equation}\label{minftyform}  
m_\infty= a \, \ell +  \sum_{y\in S}m_\infty (\{ y\})\delta_y  \; . 
\end{equation}
Here, $a$ is called the \textit{dust}, $S$ is a countable subset of $[0, x]$ that is the set of \textit{settlers} and for any $y\! \in \! S$, $m_{\infty} (\{ y\})$ is the \textit{asymptotic frequency} of the settler $y$. 

  If $a\!=\! 0$, then we say that the population $m\!:=\! (m_t)_{t\in (0, \infty)}$ has no dust (although $m_t$ may have a diffuse part at any finite time $t$).  If $a\! =\! 0$ and if 
$S$ reduces to a single point $\ree$, then $m_\infty \! =\!  \delta_\ree$ and the population $m$ 
is said to have an \textit{Eve} that is $ \ree$. Furthermore, if there exists $t_0 \! \in \! (0, \infty)$ 
such that $m_t \!= \! \delta_\ree$, for any $t\!>\!t_0$, then we say that the population has an \textit{Eve in finite time}. \qed
\end{definition}

The following theorem asserts the existence of a regular version of the frequency distributions associated with a CSBP. 
\begin{theorem}   
\label{construc} 
Let $x \!\in \!(0, \infty)$. 
Let $\Ppsi$ be a branching mechanism 
of the form (\ref{LKform}). We assume that $\Ppsi$ is not linear. Then, 
there exists a probability space $(\Omega, \ccF, \bP)$ on which the two following processes are defined. 
\begin{enumerate}
\item[(a)] $Z=(Z_t)_{t\in [0, \infty)}$ is a cadlag CSBP$\, (\Ppsi, x)$.  
\item[(b)] $M=(M_t)_{t\in [0, \infty]}$ is a $\ccM_1 ([0, x])$-valued process that is 
$\lVert \cdot \rVert_{\mathrm{var}}$-cadlag on $(0, \infty)$ such that 
$$ \forall B \in \ccB ([0, x]) , \quad  \textrm{$\bP$-a.s.~$\lim_{t\rightarrow 0+} M_t (B)= x^{-1} \ell (B)$.} $$ 
\end{enumerate} 

\noi
The processes $Z$ and $M$ satisfy the following property: 
for any Borel partition $B_1, \ldots , B_n$ of $[0, x]$ there exist $n$ independent cadlag CSBP$(\Ppsi)$, $Z^{(1)}, \ldots, \, Z^{(n)}$, 
with initial values $\ell(B_1), \ldots , \,  \ell (B_n)$, such that 
\begin{equation}
\label{Mdef}
\hspace{-10mm}\forall k\in \{ 1, \ldots , n \}\, , \;\,  \forall t \in [0, \zeta) \, , \qquad M_t (B_k)= 
Z^{(k)}_t /Z_t \; , 
\end{equation} 
where $\zeta$ stands for the time of absorption of $Z$. 
\end{theorem}

We call $M$ the \textit{frequency distribution process of a CSBP($\Ppsi, x$)}.
If $\Ppsi$ is of finite variation type, then $M$ is $\lVert \cdot \rVert_{\mathrm{var}}$-right continous at time 
$0$, which is not the case if $\Ppsi$ is of infinite variation type as explained in Section \ref{constrMsec}. 
The strong regularity of 
$M$ requires specific arguments: in the infinite variation cases, we need a decomposition  
of CSBP into Poisson clusters, which is the purpose of Theorem \ref{cluster} in Section \ref{clustersec}. 
(see this section for more details and comments).

The main result of the paper concerns the asymptotic behaviour of $M$ 
on the following three events. 
\begin{itemize}
\item $A:=\{ \zeta\! <\! \infty\}$ that is the event of absorption. Note that $\bP(A) \!>\!0$ iff $\Ppsi$ either satisfies (\ref{nonconserv}) or (\ref{nonpersist}), namely iff $\Ppsi$ is either non-conservative or non-persistent. 
\item $B:=\{ \zeta\!=\!\infty \, ;\,  \lim_{t\rightarrow \infty} Z_t \!= \! \infty\}$ that is the event of explosion in infinite time. 
Note that $\bP(B) \!>\!0$ iff $\Ppsi$ satisfies (\ref{conserv}) and $\Ppsi^\prime(0+) \!\in \! [-\infty, 0)$, namely iff 
$\Ppsi$ is conservative and super-critical. 
\item  $C:=\{ \zeta\!=\!\infty \, ; \,  \lim_{t\rightarrow \infty} Z_t \!=\! 0\}$ that is the event of extinction in 
infinite time. Note that $\bP(C) \!>\! 0$ iff $\Ppsi$ satisfies (\ref{persist}) 
and $\gamma \! <\! \infty$. 
\end{itemize}
\begin{theorem}
\label{mainth} We assume that $\Ppsi$ is a non-linear branching mechanism. 
Let $x \! \in \!(0, \infty)$ and let $M$ and $Z$ be as in Theorem \ref{construc}. Then,
$\bP$-a.s.~$\lim_{t\rightarrow \infty} \lVert M_t \!-\!M_\infty \rVert_{\mathrm{var}} \!=\! 0$, where $M_\infty$ is of the form (\ref{minftyform}). Moreover, the following holds true $\bP$-almost surely. 
\begin{itemize} 
\item[(i)] On the event $A=\{ \zeta\! <\! \infty\}$, $M$ has an Eve in finite time.  
\item[(ii)] On the event $B=\{ \zeta\!=\!\infty \, ;\,  \lim_{t\rightarrow \infty} Z_t \!= \! \infty\}$: 
\begin{itemize}
\item[(ii-a)] If $\Ppsi^\prime(0+)\!=\! -\infty$, then $M$ has an Eve; 
\item[(ii-b)] If $\Ppsi^\prime(0+) \!\in \! (-\infty, 0)$ and $\gamma \!<\! \infty$, there is no dust and $M$ has finitely many settlers whose number, under $\bP ( \, \cdot \, | B)$, is distributed as a 
Poisson r.v.~with mean $x\gamma$ conditionned to be non zero; 
\item[(ii-c)] If $\Ppsi^\prime(0+) \!\in\! (-\infty, 0)$ and $\gamma \!= \! \infty$, there is no dust and $M$ has 
infinitely many settlers that form a dense subset of $[0, x]$.  
\end{itemize}

\item[(iii)] On the event $C=\{ \zeta\!=\!\infty \, ; \,  \lim_{t\rightarrow \infty} Z_t \!=\! 0\}$: 
 \begin{itemize}
\item[(iii-a)] If $\Ppsi$ is of infinite variation type, then $M$ has an Eve; 
\item[(iii-b)] If $\Ppsi$ is of finite variation type, then the following holds true:

\begin{itemize}
\item[(iii-b-1)] If $\pi((0, 1)) \! < \! \infty$, then there is dust and $M$ has finitely many settlers whose number, under $\bP (\, \cdot \, | C)$, is distributed as a Poisson r.v.~with mean $\frac{x}{D} \int_{(0, \infty)} e^{-\gamma r} \pi (dr)$; 
\item[(iii-b-2)] If $\pi((0, 1)) \!= \! \infty$ and $\int_{(0, 1)} \pi (dr) \, r\log 1/r <\infty$, then  there is dust and there are infinitely many settlers that form a dense subset of $[0, x]$;
\item[(iii-b-3)] If $\int_{(0, 1)} \pi (dr) \, r\log 1/r =\infty$, then there is no dust and there are infinitely many settlers that form a dense subset of $[0, x]$. 
\end{itemize} 
\end{itemize} 
\end{itemize} 
\end{theorem}  
First observe that the theorem covers all the possible cases, except the deterministic ones that are trivial. 
On the absorption event $A\! =\!  \{ \zeta \! <\!  \infty\}$, the result is easy to explain: 
the descendent population of a single ancestor either explodes strictly before the others, or gets extinct strictly after the others, and there is an Eve in finite time.

The cases where there is no Eve -- namely, Theorem \ref{mainth} (\textit{ii-b}), (\textit{ii-c}) and (\textit{iii-b}) -- are simple to explain: the size of the descendent populations of the ancestors grow or decrease in the same (deterministic) scale and
the limiting measure is that of a normalised subordinator as specified in 
Proposition \ref{GreyLarge}, Lemma \ref{asLarge}, Proposition \ref{Greysmall}, Lemma \ref{assmall}, and also in the proof Section \ref{noevesec}. Let us mention that in Theorem \ref{mainth} (\textit{iii-b1}) and (\textit{iii-b2}), the dust of $M_\infty$ comes only from the dust of the $M_t$, $t\! \in \! (0, \infty)$: 
it is not due to limiting aggregations of atoms of the measures $M_t$ as $t\! \rightarrow \! \infty$.  

Theorem \ref{mainth} (\textit{ii-a}) and (\textit{iii-a}) are the main motivation of the paper: in these cases, the descendent populations of the ancestors grow or decrease in distinct scales and one dominates the others, which implies the Eve property in infinite time. This is the case of the Neveu branching mechanism $\Ppsi (\lambda) \! = \! \lambda \log \lambda$, that is related to the Bolthausen-Sznitman coalescent: see Bolthausen and Sznitman \cite{BolSzn98}, and Bertoin and Le Gall \cite{BerLeG00}.

\medskip

Let us first make some comments in connection with the Galton-Watson processes. The asymptotic behaviours displayed in Theorem \ref{mainth} (\textit{ii}) find their counterparts at the discrete level: the results of Seneta~\cite{Seneta68,Seneta69} and Heyde~\cite{Heyde70} implicitly entail that the Eve property is verified by a supercritical Galton-Watson process on the event of explosion iff the mean is infinite. However neither the extinction nor the dust find relevant counterparts at the discrete level so that Theorem \ref{mainth} (\textit{i}) and (\textit{iii}) are specific to the continuous setting.

CSBP present many similarities with generalised Fleming-Viot processes, see for instance the monograph of Etheridge~\cite{Eth00}: however for this class of measure-valued processes Bertoin and Le Gall~\cite{BerLeG03} proved that the population has an Eve without assumption on the parameter of the model (the measure $\Lambda$ which is the counterpart of the branching mechanism $\Psi$). We also mention that when the CSBP has an Eve, one can define a recursive sequence of Eves on which the residual populations concentrate, see~\cite{Labb�13}. Observe that this property is no longer true for generalised Fleming-Viot processes, see~\cite{Labb�11}.

\medskip

The paper is organized as follows. In Section \ref{estimsec}, we gather several basic properties and estimates on CSBP that are needed for the construction of the cluster measure done in 
Section \ref{clustersec}. These preliminary results are also 
used to provide a regular version of $M$ which is the purpose of Section \ref{constrMsec}. Section \ref{proofsec} is devoted to the proof of Theorem \ref{mainth}: in Section \ref{Greyresults} we state specific results on Grey martingales associated with CSBP in the cases where Grey martingales evolve in comparable deterministic scales: these results entail Theorem \ref{mainth} (\textit{ii-b}), (\textit{ii-c}) and (\textit{iii-b}), as explained in Section \ref{noevesec}. Section \ref{infineve}  
is devoted to the proof of Theorem \ref{mainth} (\textit{ii-a}) and (\textit{iii-a}): these cases 
are more difficult to handle and the proof is divided into several steps; in particular it relies on Lemma \ref{pntconv}, whose proof is postponed to Section \ref{pntconvpf}.  

\section{Construction of M.} 
\label{Msec}

\subsection{Preliminary estimates on CSBP.}
\label{estimsec}
Recall that we assume that $\Ppsi$ is not linear: namely, either $\beta \!>\!0$ or $\pi \! \neq \! 0$. 
The branching property (\ref{branchprop}) entails that for any $t \!\in \! (0, \infty)$, $u(t, \, \cdot \,)$ is the Laplace exponent of a subordinator. Namely, it is of the following form: 
\begin{equation}   
\label{uform}  
u(t, \lambda)= \kappa (t)+ d(t) \lambda + \int_{(0, \infty)} 
\!\!\! \nu_t (dr) \big( 1-e^{-\lambda r}\big)   \, , \quad \lambda \in [0, \infty)  ,  
\end{equation}
where $\kappa (t)\! = \! \lim_{\lambda \rightarrow 0+} u(t, \lambda)$,  
$d(t) \!\in \! [0, \infty)$ and $\int_{(0, \infty)} (1\wedge r) \, \nu_t (dr) \!< \! \infty$. Since $\Ppsi$ is not linear, we easily get $\nu_t \! \neq \!0$. As already mentioned in the introduction if 
$\Ppsi$ is conservative, $\kappa (t)\! =\! 0$ for any $t$ and if $\Ppsi$ is non-conservative, then  
$\kappa: (0, \infty) \! \longrightarrow \! (0, \gamma)$ is increasing and one-to-one. 
To avoid to distinguish these cases, we extend $\nu_t$ on $(0, \infty]$ by setting $ \nu_t (\{ \infty \}) \!:= \! \kappa (t)$. Thus, (\ref{uform}) can be rewritten as follows: 
$u(t, \lambda) \!=\!  d(t) \lambda + \int_{(0, \infty]} \nu_t (dr) \, (1-e^{-r\lambda} )$, with the usual convention $\exp(-\infty) \!=\! 0$.  Recall from (\ref{psifvar}) the definition of $D$. 
\begin{lemma}
\label{assuu} 
Let $t\! \in\! (0, \infty)$. Then $d(t) \!>\! 0 \, $ iff $\Ppsi$ is of finite variation type. In this case, $d(t)\!=\! e^{-Dt}$ where $D$ is defined in (\ref{psifvar}).
\end{lemma}
\begin{proof}
First note that $d(t)\! =\!  \lim_{\lambda \rightarrow \infty} \lambda^{-1} u(t, \lambda)$. An elementary computation implies that 
\begin{equation}
\label{ratio}
 \frac{u(t, \lambda)}{\lambda} = \exp \Big( \int_0^t \partial_s \, \log u \, (s, \lambda) \, ds \Big)= 
\exp \Big( -\int_0^t \frac{\Ppsi (u(s, \lambda))}{u(s, \lambda)} \, ds \Big)   \; .
\end{equation}
If $\Ppsi$ satisfies (\ref{nonpersist}), then recall that 
$\lim_{\lambda \rightarrow \infty} u(t, \lambda)\! < \! \infty$.  
Thus, in this case, $d(t)= 0$. Next assume that $\Ppsi$ satisfies (\ref{persist}). Then, 
$\lim_{\lambda \rightarrow \infty} u(t, \lambda) \!=\! \infty$. 
Note that $\Ppsi (\lambda)/\lambda$ 
increases to $\infty$ in the infinite variation cases and that it increases to the finite 
quantity $D$ in the finite variation cases, which  implies the desired result by monotone convergence in the last member of (\ref{ratio}). \cqfd 
\end{proof}

  Recall that for any $x \! \in \! [0, \infty]$, $\bbP_x$ stands for the law on $\bbD([0, \infty) , [0, \infty])$ of a CSBP($\Ppsi, x$) and recall that $\rZ$ stands for the canonical process. 
It is easy to deduce from (\ref{branchprop}) the following monotone property: 
\begin{equation}
\label{monoZ}
 \forall t \!\in\! [0, \infty), \; \forall y \!\in\! [ 0, \infty) , \; \textrm{$\forall x, x^\prime \!\in\! [0, \infty] $ such that 
$x\leq x^\prime$}, \quad \bbP_{x} \big( \rZ_t >y \big) \leq  \bbP_{x^\prime} \big( \rZ_t >y \big). 
\end{equation}
\begin{lemma}
\label{suppZ} Assume that $\Ppsi$ is not linear. Then, for all $t, x, y \!\in\! (0, \infty)$, $\bbP_x \big( \rZ_t >y \big) >0\, $. 
\end{lemma}
\begin{proof} Let $(S_x)_{x\in [0, \infty)}$ be a subordinator with Laplace exponent $u(t, \, \cdot \, )$ that is defined on an auxiliary probability space $(\Omega, \ccF, \bP)$. Thus $S_x $ under $\bP$ has the same law as $\rZ_t$ under $\bbP_x$. Since $\nu_t \! \neq \! 0$, there is $r_0\! \in\! (0, \infty)$ such that $\nu_t ((r_0, \infty))\! >\! 0$. Consequently, $N\! :=\! \# \{ z\! \in\! [0, x ] : \Delta S_z \!>\! r_0 \}$ is a Poisson 
r.v.~with non-zero mean $x \nu_t ((r_0, \infty))$. Then, for any $n$ such that $nr_0 \!>\!y$ we get 
$ \bbP_x (\rZ_t \!>\!y) \!=\! \bP (S_x \!>\!y) \! \geq\! \bP (N \geq n) \!>\!0$, which completes the proof. \cqfd 
\end{proof}
The following lemmas are used in Section \ref{clustersec} for the construction of the cluster measure. 
\begin{lemma}  
\label{flownu} Assume that $\Ppsi$ is of infinite variation type. 
Then, for any $t, s\! \in \! (0, \infty)$, 
\begin{equation}
\label{nunu}
\nu_{t+s} (dr)= \int_{(0, \infty]} \!\!\!\!\!\!  \!\!\!   \nu_s (dx) \, \bbP_x \big( \rZ_t \in dr \, ; \, \rZ_t >0  \big) \; .
\end{equation}
\end{lemma}
\begin{proof} Let $\nu$ be the measure on the right side of (\ref{nunu}). Then, for all $\lambda \! \in \! (0, \infty)$, (\ref{udef}) and (\ref{flow}) imply that 
$$ \int_{(0, \infty]} \!\!\!\!\!\!  \!\!\! \nu(dr) \big(1-e^{-\lambda r} \big)= \int_{(0, \infty]}  \!\!\!\!\!\! \!\!\! \nu_s 
(dx) \, \big( 1-e^{-x u(t, \lambda) }\big)=u(s, u(t, \lambda))= u(s+t, \lambda)=  
\int_{(0, \infty]} \!\!\!\!\!\! \!\!\! \nu_{t+s}(dr) \big(1-e^{-\lambda r} \big). $$
By letting $\lambda $ go to $0$, this implies that $\nu(\{ \infty \})= \nu_{t+s} (\{ \infty \})$. By differentiating in $\lambda$, we also get $\int_{(0,\infty)} \nu (dr) \, re^{-\lambda r}=  \int_{(0,\infty)} \nu_{t+s} (dr) \, re^{-\lambda r}$. Since Laplace transform of finite measures is injective, this entails that $\nu$ and $\nu_{t+s}$ coincide on $(0, \infty)$ which completes the proof. \cqfd  
\end{proof}
\begin{lemma}
\label{nufi} Assume that $\Ppsi$ is of infinite variation type. Then, for all $ \ee \! \in \! (0,1)$ 
and all $s, t \! \in \! (0, \infty) $ such that $s\!< \! t $, 
\begin{equation}
\label{nufin}
\int_{(0, \infty]} \!\!\! \!\! \! \nu_{s} (dx) \, \bbP_x (\rZ_{t-s} > \ee)= \nu_t \big( (\ee , \infty]  \big) \in (0, \infty) .  
\end{equation}
\end{lemma}
\begin{proof} The equality follows from (\ref{nunu}). Next observe that $\nu_t \big( (\ee , \infty]  \big) \leq 
\frac{1}{\ee} \int_{(0, \infty]} (1\wedge r) \, \nu_t (dr)< \infty$. Since $\nu_s $ does not vanish on 
$(0, \infty)$, Lemma \ref{suppZ} entails that the first member is strictly positive. \cqfd  
\end{proof}
We shall need the following simple result in the construction of $M$ in Section \ref{constrMsec}. 
\begin{lemma}
\label{pathwkcont} For all $a, y \! \in (0, \infty)$, $\lim_{r\rightarrow 0+} \bbP_r ( \rZ_a \! >\! y) \!= \! 0$ and $\lim_{r\rightarrow 0+} 
\bbP_r ( \sup_{b\in [0, a]}  \rZ_b \!  >\!  y) \! = \! 0$. 
\end{lemma}
\begin{proof} First note that  $\bbP_r ( \rZ_a \! >\! y) \leq (1\! -\! e^{-1})^{-1} \bbE_r [1\!-\! e^{- \rZ_a/y}]=   (1\! -\! e^{-1})^{-1} (1-e^{-ru(a, 1/y)}) \rightarrow 0$ as $r\rightarrow 0$, which implies the first limit. 
Let us prove the second limit: if $\gamma \!= \!\infty$, then $\rZ$ is non-decreasing and the second limit is derived from the first one. We next assume that $\gamma \! < \! \infty$, and we claim that there exist $\theta , C \! \in \!(0, 1)$ that only depend on $a$ and $y$ such that
\begin{equation} \label{pince}
 \forall z \in [y, \infty), \; \forall b \in [0, a ] , \quad p_b(z, [0, \theta y]) \leq C \; .
\end{equation}
Let us prove (\ref{pince}). 
We specify $\theta \! \in\!  (0, 1)$ further. 
By (\ref{monoZ}), $ p_b(z, [0, \theta y])\leq p_b (y , [0,\theta y] )= \bbP_y (\rZ_b \! \leq \! \theta y)$. By an elementary inequality, for all $\lambda \! \in \! (0, \infty)$, 
$\bbP_y (\rZ_b \! \leq \! \theta y) \! \leq \! \exp(y \theta  \lambda) 
\bbE_y [ \exp(\! -\! \lambda \rZ_b)] \! =\! 
\exp (y\theta  \lambda \! -\! yu(b, \lambda))$. We take $\lambda \! = \! \gamma \! +\! 1$. Thus, 
$u(\cdot , \gamma \! + \! 1)$ is decreasing and $p_b(z, [0, \theta y]) \! \leq\!  \exp ( y \theta (\gamma \! +\! 1) 
\! -\! yu(a, \gamma \!+\!1))$. We choose $\theta \! = \! \frac{u(a, \gamma +1)}{2(\gamma +1)}$. Then, (\ref{pince}) holds true with $C\! =\!  \exp (\!- \! y\theta (\gamma +1) )$. 

\smallskip

We next set $T\! = \! \inf \{ t \! \in [0, \infty) : \rZ_t \!> \! y\}$, with the convention $\inf \emptyset \! = \! \infty$. Thus $\{  \sup_{b\in [0, a]}  \rZ_b \!  >\!  y\} = \{ T\! \leq \! a \}$. Let $\theta$ and $C$ as in (\ref{pince}). First note that $ \bbP_r (T \! \leq \! a ) \! \leq \! \bbP_r ( \rZ_a \! > \! \theta y) \! +\!  \bbP_r ( T\! \leq\!  a ; 
\rZ_a \! \leq \! \theta y )$. Then, by the strong Markov property at $T$ and (\ref{pince}), we get 
$$ \bbP_r ( T\! \leq\!  a \, ; \,  
\rZ_a \! \leq \! \theta y )= \bbE_r [\un_{\{ T\leq a \} } \,  p_{a-T} (Z_T, [0, \theta y])  ] \leq C\,  \bbP_r (T\leq a )\; . $$
Thus, $\bbP_r ( \sup_{b\in [0, a]}  \rZ_b \!  >\!  y) \! \leq \! (1\! -\! C)^{-1} \bbP_r (\rZ_a \! > \! \theta y) \rightarrow 0 $ as $r \rightarrow 0$, which completes the proof.  \cqfd 
\end{proof}
We next state a more precise inequality that is used in the construction of the cluster measure of CSBP.
\begin{lemma}
\label{techZ} We assume that $\Ppsi$ is not linear. Then, for any $\ee, \eta \! \in \! (0, 1)$ 
and for any $t_0 \! \in \! (0, \infty)$, there exists $a \!\in \! (0, t_0/4)$ such that 
\begin{equation}   
\label{ineqZZ}
\hspace{-0mm} \forall x \!\in \! [0, \eta], \; \forall b \!\in\! [0, a], \; \forall c \!\in\! [ \frac{_{_1}}{^{^2}} t_0, t_0] , \quad 
 \bbP_x \Big( \!\!\! \!  \sup_{^{\; \; \; \, t\in [0, b]}} \!\! \!  \rZ_t  > 2 \eta \, ; \, \rZ_c > \ee \Big)
\leq 
2 \, \bbP_x \big( \rZ_b > \eta \, ; \, \rZ_c > \ee   \big).
\end{equation}
\end{lemma}
\begin{proof} 
Since $\Ppsi$ is not linear, $\nu_t \neq 0$. If $\gamma \! =\! \infty$, the corresponding CSBP has increasing sample paths and the lemma obviously holds true. So we assume that $\gamma \! < \! \infty$. 
We first claim the following. 
\begin{equation}
\label{posZ}
\textrm{$\forall \, x, y, t_0, t_1 \!\in \! (0, \infty)$ with $t_1\leq t_0$}, \qquad \inf_{^{t\in [t_1, t_0]}} \! 
\bbP_x \big( \rZ_t > y\big) \; >  0 .
\end{equation}
\noindent
Let us prove (\ref{posZ}). Suppose that there is a sequence $s_n \! \in \! [t_1, t_0]$ such that 
$\lim_{n \rightarrow \infty}\bbP_x (\rZ_{s_n } \!>\! y) \!= \!0$. Without loss of generality, we can assume that 
$\lim_{n \rightarrow \infty} s_n \!= \!t $. Since $u(\, \cdot \, , \lambda)$ is continuous, 
$\rZ_{s_n} \rightarrow \rZ_t $ in law under $\bbP_x$ and the Portmanteau Theorem implies that 
$\bbP_x (\rZ_t \!>\!y) \leq \liminf_{n \rightarrow \infty} \bbP_x (\rZ_{s_n} \!>\!y)= 0$, which contradicts Lemma \ref{suppZ} since $t >0$.
  
\smallskip 
  
We next claim the following: for any $\eta, \delta \!\in\! (0, 1)$, there exists $a \!\in\! (0, \infty)$ such that 
\begin{equation}  
\label{unifZ}
\forall x \! \in \!  [2\eta , \infty ) , \; \forall s \! \in \!  [0, a ] , \qquad \bbP_x \big( \rZ_s \leq \eta \big) \leq \delta .   
\end{equation}
\noindent
Let us prove (\ref{unifZ}). We fix $x\in [2\eta , \infty)$. Let $a \!\in \!(0, \infty)$ that is specified later. For any $s \!\in\! [0, a]$, the Markov inequality entails for any $\lambda \!\in\! (0, \infty)$
\begin{equation}
\label{Mrkv}
 \bbP_x (\rZ_s \leq \eta)\leq e^{\lambda \eta } \bbE_x \big[ e^{-\lambda \rZ_s}\big]= 
e^{\lambda \eta -xu(s, \lambda)  }
\leq 
e^{- \lambda \eta + 2 \eta (\lambda \!- \!u(s, \lambda)) }\; .
\end{equation}
We now take $\lambda \! >\! \gamma$. Then, $u(\, \cdot \, , \lambda)$ is decreasing and we get 
\begin{equation}
\label{uspec}
  -\lambda \eta + 2 \eta \big(\lambda \!- \! u(s, \lambda)\big) \leq  -\lambda \eta + 2 \eta \! \int_0^s \!\! \Ppsi (u(b, \lambda)) \, d b \leq -\lambda \eta + 2  \eta \, a  \Ppsi (\lambda) .
\end{equation}  
Then set $\lambda= \gamma +1\!-\! \eta^{-1} \log \delta$ and $a= (\gamma +1) /(2 \Ppsi (\lambda))$, which entails (\ref{unifZ}) by (\ref{uspec}) and (\ref{Mrkv}).

\smallskip

We now complete the proof of the lemma. We first fix $\ee, \eta \! \in \! (0, 1) $ and $t_0 \! \in\!  (0, \infty)$ and then we set
$$ \delta \, = \, \frac{1}{2} \, \frac{\, \inf_{^{t\in [\frac{_1}{^4}t_0, t_0]}}  
\bbP_{2\eta} \big( \rZ_t > \ee\big)\,  }{\, \sup_{^{t\in [\frac{_1}{^4}t_0, t_0]}}  \! 
\bbP_{\eta} \big( \rZ_t > \ee \big) \, } \; \, .$$
By (\ref{posZ}), $\delta \!>\!0$. Let $a \! \in \! (0, \frac{_{_1}}{^{^4}} t_0)$ be such that (\ref{unifZ}) holds true. We then fix $x \!\in \! [0, \eta]$, $b\! \in \! [0, a]$ and $c\!\in \! [\frac{_{_1}}{^{^2}} t_0 , t_0]$ and we introduce the stopping time $T= \inf \{ t \!\in \! [0, \infty) : \rZ_t > 2\eta \}$. Then, 
\begin{equation}     
\label{Adef}
A:= \bbP_x \Big( \!\!\! \!  \sup_{^{\; \; \; \, s\in [0, b]}} \!\! \!  \rZ_t  > 2 \eta \, ; \, \rZ_c > \ee \Big)= \bbP_x 
\big( T \leq b \, ; \, \rZ_c \! >\! \ee\big) \leq \bbP_x \big(\rZ_b \! >\, \eta \, ; \, \rZ_c \! > \! \ee  \big) + B \; , 
\end{equation}
where $B:=  \bbP_x \big( T \leq b \, ; \, \rZ_b \leq \eta \, ; \, \rZ_c \! >\! \ee\big)$ is bounded as follows: by the Markov property at time $b$ and by (\ref{monoZ}), we first get 
$$ B \leq \bbE_x \big[ \un_{\{ T \leq b \, ; \, \rZ_b \leq \eta  \}} \,  \bbP_{\rZ_b}( \rZ_{c-b} \! >\! \ee   ) \big] 
\leq \bbP_{\eta} ( \rZ_{c-b} \! >\! \ee ) \, \bbE_x \big[ \un_{\{ T \leq b \, ; \, \rZ_b \leq \eta  \}} \big]. $$
Recall that $p_t (x,  dy  )\! = \! \bbP_x (\rZ_t \!\in\! dy)$ stands for the transition kernels of $\rZ$. The strong Markov property at time $T$ then entails   
$$ \bbE_x \big[ \un_{\{ T \leq b \, ; \, \rZ_b \leq \eta  \}} \big]= \bbE_x \big[ \un_{\{ T \leq b \}} \, p_{ b-T} \big( \rZ_T , [0, \eta ]\big) \big] .$$
Next observe that $\bbP_x$-a.s.~$b\!-\!T \!\leq \!a$ and $\rZ_T \!> \!2 \eta$, which implies 
$p_{b-T} \big( \rZ_T , [0, \eta ]\big) \leq \delta $ by (\ref{unifZ}). Thus, 
$$ B  \leq \delta \, \bbP_{\eta} ( \rZ_{c-b} \! >\! \ee ) \, \bbE_x \big[ \un_{\{ T \leq b \}} \big] \; .$$
Since $c\! -\! b \! \in \![\frac{_1}{^4}t_0, t_0]$, we get 
$\delta \, \bbP_{\eta} ( \rZ_{c-b} \! >\! \ee ) \leq 
\frac{1}{2} \inf_{{t\in [\frac{_1}{^4}t_0, t_0]}}  \bbP_{2\eta} \big( \rZ_t  \!>\!  \ee \big)$, by definition of $\delta$. Next, observe that 
$$ \textrm{$\bbP_x$-a.s.~on $\{ T\leq b\}$}, \quad  \inf_{^{t\in [\frac{_1}{^4}t_0, t_0]}}  \bbP_{2\eta} \big( \rZ_t  \!>\!  \ee \big) \leq p_{c-T} \big( 2\eta , (\ee , \infty] \big)  \leq p_{c-T} \big( \rZ_T , (\ee , \infty] \big) , $$
where we use (\ref{monoZ}) in the last inequality. Thus, by the strong Markov property at time $T$ and the previous inequalities, we finally get 
$$ B \leq  \frac{_1}{^2} \bbE_x \big[ \un_{\{ T \leq b \}}\, p_{c-T} \big( \rZ_T , (\ee , \infty] \big) \big] =  \frac{_1}{^2} \bbP_x 
\big( T \leq b \, ; \, \rZ_c \! >\! \ee\big)= \frac{_1}{^2} A \; , $$
which implies  the desired result by (\ref{Adef}). \cqfd

\end{proof}

We end the section by a coupling of finite variation CSBP. To that end,   
let us briefly recall that CSBP are time-changed L\'evy processes via Lamperti transform: 
let $X=(X_t)_{t\in [0, \infty)}$ be a cadlag L\'evy process without negative jump that is 
defined on the probability space 
$(\Omega, \cF, \bP)$. We assume that $X_0\!=\!x \!\in \! (0, \infty)$ and that 
$\bE[\exp (-\lambda X_t)] \!=\! \exp (-x\lambda + t\Ppsi (\lambda))$. We then set 
\begin{equation}\label{Lampe} 
\tau= \inf \big\{ t\! \in\!  [0, \infty): X_t\! =\! 0 \big\}, \quad L_t = \tau \wedge \inf \Big\{ s\!\in\! [0, \tau): \int^s_0\! \frac{dr}{X_r} > t  \Big\}
\quad \textrm{and} \quad Z_t= X_{ L_t },   
\end{equation}  
with the conventions $\inf \emptyset \!= \!\infty$ and $X_\infty\!=\! \infty$. Then, $(Z_t)_{t\in [0, \infty)}$ is a CSBP($\Ppsi, x$). See \cite{CabLamUri09} for more details. 
Recall from (\ref{psifvar}) the definition of $D$. 
\begin{lemma}
\label{supZfinvar} Assume that $\Ppsi$ is of 
finite variation type and that $D$ is strictly positive.  
Let $(Z_t)_{t\in [0, \infty)}$ be a CSBP$(\Ppsi, x)$ defined on a probability space $(\Omega, \cF, \bP)$. 
For any $\lambda \in [0, \infty)$, set $\Ppsi^*(\lambda):= \Psi (\lambda)-D\lambda$.
Then, there exists $(Z^*_t)_{t\in [0, \infty)}$, a CSBP$(\Ppsi^*, x)$ on $(\Omega, \cF, \bP)$ such that 
$$ \textrm{$\bP$-a.s.} \quad \forall t \in [0, \infty), \quad \sup_{s\in [0, t]} Z_t  \leq Z^*_t \; .$$
\end{lemma}
\begin{proof} Without loss of generality, we assume that there exists a L\'evy process $X$ defined on $(\Omega, \cF, \bP)$ such that $Z$ is derived from $X$ by the Lamperti time-change (\ref{Lampe}). 
We then set $X^*_t \!=\! X_t\! +\!Dt$ that is a subordinator with Laplace exponent $-\Ppsi^*$ and with initial value $x$. Since $D$ is positive, we have $X_t \leq X^*_t$ for all $t \!\in\! [0, \infty)$. Observe that $\tau^*\!=\!\infty$. Let $L^*$ and $Z^*$ be derived from $X^*$ as $L$ and $Z$ are derived from $X$ in (\ref{Lampe}). Then, $Z^*$ is a 
CSBP($\Ppsi^*, x$) and observe that $L^*_t \! \geq \!L_t$. Since $X^*$ is non-decreasing, $Z^*_t \!=\! 
X^*_{L^*_t} \! \geq \! X^*_{L_t}\! \geq \!X_{L_t}\!=\! Z_t$, which easily implies the desired result since $Z^*$ is non-decreasing. \cqfd
\end{proof}

\subsection{The cluster measure of CSBP with infinite variation.}
\label{clustersec}
Recall that $\bbD([0, \infty) , [0, \infty])$ stands for the space of $[0, \infty]$-valued cadlag functions. Recall that $\rZ$ stands for the canonical process. For any $t \! \in\! [0, \infty)$, we denote by $\ccF_t$ the canonical filtration. Recall from (\ref{absordef}) the definition of the times of absorption 
$\zeta_0$, $\zeta_\infty$ and $\zeta$. Also recall from the beginning of Section \ref{estimsec} the definition of the measure $\nu_t$ on $(0, \infty]$.  
\begin{theorem}
\label{cluster} Let $\Ppsi$ be of infinite variation type. Then, there exists a unique $\sigma$-finite measure 
$\rN_\Ppsi$ on $\bbD([0, \infty) , [0, \infty])$ that satisfies the following properties. 
\begin{enumerate}
\item[{\rm (a)}] $\rN_\Ppsi$-a.e.~$\rZ_0= 0$ and $\zeta >0$.     

\item[{\rm (b)}] $\nu_t (dr) = \rN_\Ppsi \big( \rZ_t \! \in \! dr \, ; \, \rZ_t >0 \big)$, for any $t\in (0, \infty)$.   

\item[{\rm (c)}] $ \rN_\Ppsi \big[ F(\rZ_{\, \cdot \, \wedge t } )\, G( \rZ_{t + \, \cdot \, } ) \,; \,   \rZ_{t} >0 \, 
\big]=  
\rN_\Ppsi \big[ F(\rZ_{\, \cdot \, \wedge t } ) \, \bbE_{\rZ_t} [ \, G \, ] 
\,; \,   \rZ_t >0 \, \big] $, for any nonnegative functionals $F, G$ and for any $t \in (0, \infty)$.  
\end{enumerate}
The measure $\rN_\Ppsi$ is called the cluster measure of CSBP$(\Ppsi)$. 
\end{theorem}
\begin{comment}
\label{biblioN} The existence of $\rN_\Ppsi$ - sometimes called Kuznetsov measure, see~\cite{Kuz73} - is not really new: for sub-critical $\Ppsi$, $\rN_\Ppsi$ can be derived from the excursion measure of the height process of the L\'evy trees and the corresponding 
super-processes as introduced in \cite{DuqLeG02}. See also Dynkin and Kuznetsov~\cite{DynKuz04} for a different approach on super-processes. We also point out articles of Li~\cite{Li01,LiBook,LiNotes12} on the construction of this measure when $\Ppsi'(0+) \ne -\infty$. Here, we provide a brief and self-contained proof of the existence of the cluster measure for CSBP that works in all cases. \cq 
\end{comment}
\begin{proof} The only technical point to clear is (a): namely, the right-continuity at time $0$.
For any $s, t \! \in \! (0, \infty)$ such that $s\!\leq \!t $ and for any $\ee \! \in \!(0, 1)$, we define a measure $Q_{t, \ee}^s $ on $\bbD([0, \infty) , [0, \infty])$ by setting 
\begin{equation}
\label{Qdef}
Q_{t, \ee}^s [\, F \, ] = \frac{1}{\nu_t ((\ee, \infty])} \int_{(0, \infty]} \!\!\!\!\! \! \! \, \nu_s (dx) \, \, \bbE_x 
\big[  F( \rZ_{(\, \cdot\,  -s)_+} ) \, ; \, \rZ_{t-s} \!> \!\ee \, \big]  \, , 
\end{equation}
for any functional $F$. By Lemma \ref{nufi}, (\ref{Qdef}) makes sense and it defines a probability measure on the space $\bbD([0, \infty) , [0, \infty])$. The Markov property for CSBP and Lemma \ref{nufi} easily 
imply that for any $s\! \leq\!  s_0\! \leq\! t$, 
\begin{equation}
\label{Qfix}
Q_{t, \ee}^s \big[ F (\rZ_{s_0 + \, \cdot \, } )  \big] = \frac{1}{\nu_t ((\ee, \infty])} 
\int_{(0, \infty]} \!\!\!\!\! \! \! \, \nu_{s_0} (dx) \, \, \bbE_x 
\big[  F( \rZ) \, ; \, \rZ_{t-s_0} \!> \!\ee \, \big]  \, , 
\end{equation}
We first prove that for $t$ and $ \varepsilon$ fixed, the laws $Q^s_{t, \ee}$ are tight as 
$s\! \rightarrow \!0$. By (\ref{Qfix}), it is clear that we only need to control the paths in a neighbourhood of time $0$. By a standard criterion for Skorohod topology (see for instance Theorem 16.8~\cite{Billin99}  p.~175), the laws $Q^s_{t, \ee}$ are tight as $s\! \rightarrow\! 0$ if the following claim 
holds true: for any $\eta , \delta \! \in\!  (0, 1)$, there exists $a_1 \! \in \! (0, \frac{_{_1}}{^{^4}} t)$ such that 
\begin{equation}
\label{claimQ}
\forall s \!\in \! (0, a_1 ], \quad Q^s_{t, \ee} \Big( \sup_{^{[0, a_1 ]}} \rZ > 2\eta \Big) < \delta \; .
\end{equation} 
To prove (\ref{claimQ}), we first prove that for any $\eta , \delta \! \in \! (0, 1)$, there exists $a_0  \in 
(0,t)$ such that 
\begin{equation}
\label{claimQQ}
\textrm{$\forall s , b \!\in \! (0, a_0  ]$ such that $s\leq b$}, 
\quad Q^s_{t, \ee} \big( \rZ_b  \!>\!  \eta \big) < \frac{_1}{^3} \delta \; .
\end{equation} 
\noindent
\textit{Proof of (\ref{claimQQ}).} Recall that $\un_{[1, \infty]} (y) \! \leq\! C (1\!-\!e^{-y}) $, for any $y \!\in \![0, \infty]$, where $C\!=\! (1\!-\!e^{-1})^{-1}$. 
Fix $\eta , \delta \! \in \! (0, 1) $ and $s, b\! \in\! (0, t)$ such that $s \!\leq \!b$. 
Then, (\ref{Qfix}), with $b\!=\! s_0$, implies that 
\begin{eqnarray*} 
Q^s_{t, \ee} \big( \rZ_b  \!>\!  \eta \big) & \leq & C\, Q^s_{t, \ee} \big( 1-e^{-\frac{1}{\eta} \rZ_b} \big)= 
\frac{C}{\nu_t ((\ee, \infty])} \int_{(0, \infty]} \!\!\!\!\! \! \! \, \nu_{b} (dx) \, \big( 1-e^{-\frac{1}{\eta} x} \big) \, 
\bbP_x ( \rZ_{t-b} \!> \!\ee ) \\
& \leq &\frac{C^2}{\nu_t ((\ee, \infty])}
\int_{(0, \infty]} \!\!\!\!\! \! \! \, \nu_{b} (dx) \, \big( 1-e^{-\frac{1}{\eta} x} \big) \, 
\bbE_x \big[  1-e^{-\frac{1}{\ee} \rZ_{t-b}} \big] \\
& \leq &  \frac{C^2}{\nu_t ((\ee, \infty])}
\int_{(0, \infty]} \!\!\!\!\! \! \! \, \nu_{b} (dx) \, \big( 1-e^{-\frac{1}{\eta} x} \big) \big( 1-e^{-x u(t-b,\frac{1}{\ee}) }
\big)= : f(b) . 
\end{eqnarray*}
By developping the product in the integral of the last right member of the inequality, we get 
$$ f(b)=  \frac{C^2}{\nu_t ((\ee, \infty])} \Big(  u(b, \frac{_1}{^\eta}) + u(t, \frac{_1}{^\ee}) -u\big( \, b\, ,\,  \frac{_1}{^\eta}\!+ \!u(t\!-\!b, \frac{_1}{^\ee}) \big)  \, \Big) \; \underset{^{b \rightarrow 0}}{-\!\!-\!\! \!\longrightarrow} \; 0 \, ,$$ 
We then define $a_0$ such that $\sup_{b\in (0, a_0]} f(b) < \frac{1}{3} \delta$, which implies (\ref{claimQQ}). \qed 

\smallskip

\noindent
\textit{Proof of (\ref{claimQ}).} We fix  $\eta , \delta \! \in \!(0, 1) $. Let $a \! \in \! (0, \frac{_{_1}}{^{^4}} t)$ 
such that (\ref{ineqZZ}) in Lemma \ref{techZ} holds true with $t_0\! = \!t$. 
Let $a_0$ as in (\ref{claimQQ}). We next set 
$a_1 \! =\!  a\! \wedge \! a_0$. We fix $s \! \in\!  (0, a_1]$ and we then get the following inequalities: 
\begin{eqnarray*} 
Q^s_{t, \ee} \Big( \sup_{^{[0, a_1 ]}} \rZ > 2\eta \Big) & \leq & Q^s_{t, \ee} (\rZ_s > \eta )+  
Q^s_{t, \ee} \Big( \sup_{^{[0, a_1 ]}} \rZ > 2\eta  \, ; \  \rZ_s \leq  \eta  ,\Big) \\
& \leq &  \frac{_1}{^3} \delta + \frac{1}{\nu_t ((\ee, \infty])} \int_{(0, \eta]} \!\!\!\!\! \! \! \, \nu_{s} (dx) \, 
\bbP_x  \Big( \!\! \sup_{\; \;^{[0, a_1-s ]}} \!\! \rZ > 2\eta  \, ; \  \rZ_{t-s} \! >\! \ee  \Big) \\
& \leq & \frac{_1}{^3} \delta + \frac{2}{\nu_t ((\ee, \infty])} \int_{(0, \eta]} \!\!\!\!\! \! \! \, \nu_{s} (dx) \, 
\bbP_x  \big(\rZ_{a_1-s} > \eta  \, ; \  \rZ_{t-s} \! >\! \ee  \big)\\
& \leq &  \frac{_1}{^3} \delta  + 2\, Q^s_{t, \ee} (\rZ_{a_1} > \eta ) \; < \; \delta .
 \end{eqnarray*}
Here we use (\ref{claimQQ}) in the second line, (\ref{ineqZZ}) in the third line and (\ref{claimQQ}) in the fourth one. \qed  

\smallskip
  
We have proved that for $t, \varepsilon$ fixed, the laws $Q^s_{t, \ee}$ are tight as $s\! \rightarrow \! 0$. 
Let $Q_{t, \ee}$ stand for a possible limiting law. By a simple argument, $Q_{t, \varepsilon}$ has no fixed 
jump at time $s_0$ and basic continuity results entail that (\ref{Qfix}) holds true with $Q_{t,\ee}$ instead of $Q^s_{t, \ee}$, which fixes the finite-dimensional marginal laws of $Q_{t, \ee}$ on $(0, \infty)$. Next observe that for $\eta, \delta \! \in \! (0, 1)$ and $a_1 \!  \in \!  (0, \frac{_{_1}}{^{^4}} t)$ as 
in (\ref{claimQ}), the set 
$\{ \sup_{(0, a_1)} \rZ \! > \! 2\eta \}$ is an open set of $\bbD([0, \infty), [0, \infty])$. Then, by (\ref{claimQ}) and the Portmanteau Theorem, $Q_{t, \ee} (\sup_{(0, a_1)} \rZ > 2 \eta) \leq  \delta$. This easily implies that 
$Q_{t, \ee}$-a.s.~$\rZ_0= 0$, which completely fixes the finite-dimensional marginal laws of 
$Q_{t, \ee}$ on $[0, \infty)$. This proves that there is only one limiting distribution and $Q^s_{t, \ee} \rightarrow Q_{t, \ee}$ in law as $s \rightarrow 0$. 

We next set $\rN_{t, \ee}= \nu_t ((\ee, \infty]) \, Q_{t, \ee}$. We easily get 
$\rN_{t, \ee} -\rN_{t, \ee^\prime}= \rN_{t, \ee} (\, \cdot \, \, ; \, \rZ_t \!\in\! (\ee, \ee^\prime] )$, for any $0\! <\! \ee \!< \! \ee^\prime \!< \! 1$. Fix $\ee_p \in (0, 1)$, $p\in \bbN$, that decreases to $0$. 
We define a measure $\rN_t$ by setting 
$$ \rN_t =\rN_{t, \ee_0} + \sum_{p\geq 0} \; \rN_{t, \ee_{p\!+\!1}} \big( \, \cdot \, \, ; \, \rZ_t \! \in \! (\ee_{p\!+\!1} , \ee_p] \, \big) = \rN_{t, \ee_0} + \sum_{p\geq 0} \; \rN_{t, \ee_{p\!+\!1}} \!- \!  \rN_{t, \ee_{p}}  \; .$$
By the first equality, $\rN_t$ is a well-defined $\sigma$-finite measure; the second equality shows 
that the definition of $\rN_t$ does not depend on the sequence $(\ee_p)_{p\in \bN}$, which implies 
$\rN_t  \big( \, \cdot \, \, ; \, \rZ_t > \ee \, \big) = \nu_t ((\ee, \infty]) \, Q_{t, \ee} $, for any $\ee \! \in \! (0, 1)$. 
Consequently, we get $\rN_t -\rN_{t^\prime}= \rN_t (\, \cdot \, \, ; \, \rZ_{t^\prime}= 0 )$, for any $t^\prime \!> \! t \! >\! 0$. Fix $t_q \in (0, 1)$, $q\in \bbN$, that decreases to $0$. 
We define $\rN_\Ppsi$ by setting 
$$ \rN_\Ppsi = \rN_{t_0}+ \sum_{q\geq 0} \; \rN_{t_{q\!+\!1}} \big( \, \cdot \, \, ; \, \rZ_{t_q}=0  \, \big) 
= \rN_{t_0} + \sum_{q\geq 0} \; \rN_{t_{q\!+\!1}} \!- \!  \rN_{t_{q}}  \; .$$
The first equality shows that $\rN_\Ppsi$ is a well-defined measure and the second one that its definition does not depend on the sequence $(t_q)_{q\in \bN}$, which implies 
 \begin{equation}
\label{conct}
\forall \ee \! \in \! (0, 1) , \; \forall t \in (0, \infty) , \quad \rN_\Ppsi  \big( \, \cdot \, \, ; \, \rZ_t > \ee \, \big) = 
\nu_t ((\ee, \infty]) \, Q_{t, \ee} \; .
\end{equation}
This easily entails that for any nonnegative functional $F$ 
\begin{equation}
\label{shift}
 \forall t \in (0, \infty) , \quad \rN_\Ppsi \big[ F( \rZ_{t + \, \cdot \, } ) \, ; \, \rZ_t \! > \! 0 \big]= \int_{(0, \infty]} \!\!\!\!\!\! \nu_t (dx) \, \bbE_x [\, F  \, ] \; .
\end{equation}
Recall that $\zeta $ is the time of absorption in $\{ 0, \infty\}$. Since $\rN_{t_q, \ee_p} (\zeta \!=\!  0)\!= \!0$, we get $\rN_\Ppsi (\zeta \!=\! 0)\!=\! 0$ and thus, $\rN_\Ppsi (\{ \mathtt{O} \})\!=\! 0$, where 
$\mathtt{O}$ stands for the null function. 
Set $A_{p, q}= \{ \rZ_{t_q} \!> \! \ee_p\}$. Then, $\rN_\Ppsi (A_{p,q}) \!<\!  \infty$ by (\ref{conct}). Since
$\bbD([0, \infty) , [0, \infty])\!=\! \{ \mathtt{O}\} \cup \bigcup_{p,q\geq 1} A_{p, q}$, $\rN_\Ppsi$ is sigma-finite. Properties (b) and (c) are easily derived from (\ref{shift}), 
(\ref{conct}) and standard limit-procedures: the details are left to the reader. \cqfd
\end{proof}

\subsection{Proof of Theorem \ref{construc}.}
\label{constrMsec}
\subsubsection{Poisson decomposition of CSBP.} 
From now on, we fix $(\Omega, \ccF, \bP)$, a probability space on which are defined all the random variables that we mention, 
unless the contrary is explicitly specified. We also fix $x \!\in \!(0, \infty)$ and we recall that $\ell$ stands for the Lebesgue measure on $\bbR$ or on $[0, x]$, according to the context.   

We first briefly recall Palm formula for Poisson point measures: let $E$ be a Polish space equipped with its Borel sigma-field $\ccE$. Let $A_n \! \in\! \ccE$, $n\! \in \! \bbN$, be a partition of 
$E$. We denote 
by $\ccM_{\textrm{pt}} (E)$ the set of point measures $m$ on $E$ such that $m(A_n)\! < \! \infty$ for any $n\! \in \!  \bbN$; we equip $\ccM_{\textrm{pt}} (E)$ with the sigma-field generated by the applications $m\mapsto m(A)$, where $A$ ranges in $\cE$. Let $\cN \! =\!  \sum_{i\in I} \delta_{z_i}$ be a Poisson point measure on 
$E$ whose intensity measure $\mu$ satisfies $\mu(A_n)\! < \! \infty$ for every $n\! \in \!  \bbN$. 
We shall refer to the following as the \textit{Palm formula}:
for any measurable $F:E \! \times \!  \ccM_{\textrm{pt}} (E) \!  \longrightarrow \!  [0, \infty)$, 
\begin{equation}  
\label{Palm}
\bE \Big[ \sum_{i\in I} F(z_i \, , \cN\!-\!\delta_{z_i} ) \Big]= \int_E \! \mu(dz) \, \bE \big[ F(z \, , \cN) \big] \; .
\end{equation}
If one applies twice this formula, then we get for any measurable $F:E \! \times \!  E \! \times \! \ccM_{\textrm{pt}} (E) \!  \longrightarrow \!  [0, \infty)$, 
\begin{equation}  
\label{rePalm}
 \bE \Big[ \sum_{\substack{i, j\in I \\ i\neq j}} F(z_i \, ,  z_j\, ,  \cN\!-\!\delta_{z_i}\!-\!\delta_{z_j} ) \Big]= \int_E \! \mu(dz) \int_E \! \mu(dz^\prime) \, \bE \big[ F(z \, , z^\prime \, ,  \cN) \big] \; .
\end{equation}
We next introduce the Poisson point measures that are used to define the population associated with a  CSBP. 

\smallskip

\noi
\textbf{Infinite variation cases.} We assume that $\Ppsi$ is of infinite variation type. Let 
\begin{equation} \label{defP}
\ccP= \sum_{i\in I} \delta_{(x_i, \rZ^i)} \end{equation} 
be a Poisson point measure on $[0, x] \! \times \! \bbD([0, \infty), [0, \infty])$, with intensity $\un_{[0, x]}(y) \ell(dy)\rN_\Ppsi (d\rZ)$, where $\rN_\Ppsi$ is the cluster measure associated with $\Ppsi$ as specified in Theorem \ref{cluster}. Then, for any $t \! \in \! (0, \infty)$, we define the following random point measures on $[0, x]$:    
\begin{equation}
\label{cZdef}
 \cZ_t   =  \sum_{i\in I} \rZ^i_t \, \delta_{x_i}  \quad \textrm{and}  \quad
\cZ_{t-}   =  \sum_{i\in I} \rZ^i_{t-} \, \delta_{x_i}    \; . 
\end{equation}
We also set $\cZ_0 =  \ell (\, \cdot \, \cap [0, x])$. \cq

\medskip

\noi
\textbf{Finite variation cases.} We assume that $\Ppsi$ is of finite variation type and not linear. 
Recall from (\ref{psifvar}) the definition of $D$. Let 
\begin{equation} \label{defQ}
\ccQ= \sum_{j\in J} \delta_{(x_j, t_j,\rZ^j)} 
\end{equation} 
be a Poisson point measure on $[0, x] \! \times \!  [0, \infty) \! \times \! \bbD([0, \infty), [0, \infty])$, whose intensity measure is  
$$\un_{[0, x]}(y) \ell(dy) \, e^{-Dt} \ell (dt) \int_{(0,\infty)} \!\!\!\!\!\!\!\! \pi(dr)\,  \bbP_r ( d\rZ) \, , $$ 
where $\bbP_r$ is the canonical law of a CSBP($\Ppsi, r$) and $\pi$ is the L\'evy measure of $\Ppsi$. Then, for any $t \! \in \! (0, \infty)$, we define the following random measures on $[0, x]$:    
\begin{equation}
\label{cZdef2}
 \cZ_t   \!= \!e^{-Dt} \ell (\, \cdot \, \cap [0, x]) \!+\! \sum_{j\in J} \un_{\{ t_j \leq t \}} 
 \rZ^j_{t-t_j}  \delta_{x_j} ,  \quad \cZ_{t-}\!=\! e^{-Dt} \ell (\, \cdot \, \cap [0, x]) \!+\! \sum_{j\in J} 
 \un_{\{ t_j \leq t \}} 
 \rZ^j_{(t-t_j)-}  \delta_{x_j} .
\end{equation}
We also set $\cZ_0 =  \ell (\, \cdot \, \cap [0, x])$.  \cq

\medskip

 In both cases, for any $t\! \in \! [0, \infty)$ and any $B \in \ccB([0, x])$, 
 $\cZ_t (B) $ and $\cZ_{t-} (B)$ are $[0, \infty]$-valued $\ccF$-measurable random variables. 
The finite dimensional marginals of 
$(\cZ_t (B))_{t\in [0, \infty)}$ are those of a CSBP($\Ppsi, \ell (B)$): in the infinite variation cases, it is a simple consequence of Theorem \ref{cluster} (c); in the finite variation cases, 
it comes from direct computations: we leave the details to the reader. 
Moreover, if $B_1, \ldots, \, B_n$ are disjoint Borel subsets of $[0,x]$, note that 
the processes $(\cZ_t (B_k))_{t\in [0, \infty)}$, $1\! \leq \! k \! \leq\! n$ are independent. 
To simplify notation, we also set 
\begin{equation} \label{defZZ}
\forall t\in [0, \infty ), \quad Z_t= \cZ_t ([0, x]) \; , 
 \end{equation}
that has the finite dimensional marginals of a CSBP($\Ppsi, x$).  

\subsubsection{Regularity of $\cZ$.} 
Since we deal with possibly infinite measures, we introduce the following specific notions. We fix 
a metric $d$ on $[0, \infty]$ that generates its topology. 
For any positive Borel measures $\mu$ and $\nu$ 
on $[0, x]$, we define their variation distance by setting 
\begin{equation}
\label{dvar}
 d_{{\rm var}} (\mu, \nu):=\!\!\!\!\!  \sup_{\; \; \; \; ^{B \in \ccB([0, x])}} \!\!\!\!\! 
d \big( \mu (B)  ,  \nu (B)\big) \; .
\end{equation}
The following proposition deals with the regularity of $\cZ$  on $(0, \infty)$, which is sufficient for our purpose. The regularity at time $0$ is briefly discussed later.
\begin{proposition}
\label{reguZ} Let $\cZ$ be as in (\ref{cZdef}) or (\ref{cZdef2}). Then, 
\begin{equation}
\label{cZreg}
 \textrm{$\bP$-a.s.~$\; \forall t\in (0, \infty)$,} \quad  \lim_{h \rightarrow 0+} d_{{\rm var}} \big( \cZ_{t+h} , \cZ_t 
 \big)= 0 \quad {\rm and} \quad    \lim_{h \rightarrow 0+} d_{{\rm var}} \big( \cZ_{t-h} , \cZ_{t-} 
 \big)= 0.
\end{equation}
\end{proposition}    
 \begin{proof} We first prove the infinite variation cases. We proceed by approximation. Let us fix $s_0 \in (0, \infty)$. For any $\ee \in (0, 1)$, we set 
$$ \forall t \in (0, \infty) \, , \quad \cZ_t^\ee= \sum_{{i\in I}} \un_{\{ \rZ^i_{s_0} >\ee \}} \, \rZ^i_t \, \delta_{x_i} .   
$$
Note that $\#\{ i \!\in\! I : \rZ^i_{s_0} \!>\! \ee \}$ is a Poisson r.v.~with mean $x\rN_\Ppsi (\rZ_{s_0} \!>\!\ee) = x \nu_{s_0} ((\ee, \infty]) \!<\! \infty$. Therefore, $\cZ^\ee$ is a finite sum of weighted Dirac masses whose weights are cadlag $[0, \infty]$-valued processes. Then, by an easy argument, $\bP$-a.s.~$\cZ^\ee$ is $d_{{\rm var}}$-cadlag on $(0, \infty)$.   
 
For any $v \!\in \![0, \infty]$, then set $\varphi (v)\!= \! \sup \{ d(y, z) \, ; \, y \!\leq \! z \! \leq \! y\!+\!v \}$, which is well-defined, bounded, non-decreasing and such that $\lim_{v\rightarrow 0} \varphi (v)=0$. For any 
$ \varepsilon \! >\! \ee^\prime \! > \! 0$, observe that $\cZ_t^{\ee^\prime}\!= \!\cZ_t^{\ee}+ \sum_{i\in I}  \un_{\{ \rZ^i_{s_0}  \in  (\ee^\prime ,\ee] \}} \, \rZ^i_t \, \delta_{x_i}$. Then, we fix $T \! \in \! (0, \infty)$, we set $Y_{t}^{\ee^\prime \! \!, \,  \ee}:= \sum_{i\in I}  \un_{\{ \rZ^i_{s_0}  \in  (\ee^\prime ,\ee] \}} \, \rZ^i_{s_0+t}$ and we get 
$$ \sup_{^{t \in [s_0, s_0 +T]}} \!\! d_{{\rm var}} \big(\cZ_t^{\ee^\prime}, \cZ_t^{\ee} \big) \leq  \varphi (V_{\ee^\prime \! \!, \,  \ee}) \quad \textrm{where} \quad V_{\ee^\prime \! \!, \,  \ee}:= \sup_{t\in [0, T]} Y^{\ee^\prime \! \!, \,  \ee}_t \; .$$
Note that $ Y^{\ee^\prime \! \!, \,  \ee}$ is a cadlag CSBP($\Ppsi$). The exponential formula for Poisson point measures and Theorem \ref{cluster} (b) imply for any $\lambda \! \in \! (0, \infty)$, 
$$ -\frac{_1}{^x} \log \bE \big[ \exp \big(\! -\! \lambda Y^{\ee^\prime \! \!, \,  \ee}_0 \big)\big]= \int_{(\ee^\prime , \ee]} \!\!\!\!\!\!\!  \nu_{s_0} (dr) \big( 1-e^{-\lambda r}\big)\leq \lambda \! \int_{(0 , \ee]} \!\!\!\!\!\!\!  \nu_{s_0} (dr) \, r \underset{{\varepsilon \rightarrow 0}}{-\!\!\!-\!\!\! \longrightarrow } 0 \; .$$
For any $\eta \! \in \! (0, \infty)$, it easily implies $\lim_{ \varepsilon \rightarrow 0} \sup_{\ee^\prime  \in (0,  \varepsilon]} \bP ( Y_{0}^{\ee^\prime \! \!, \,  \ee}  \! >\! \eta ) \!= \! 0$. 
Next, note that $r\! \longmapsto \! \bbP_r ( \sup_{t\in [0, T]}  \rZ_t \! > \! \eta )$ is non-decreasing and recall that $\lim_{r\rightarrow 0+}  \bbP_r ( \sup_{t\in [0, T]}  \rZ_t \! > \! \eta )\!=\! 0$, by Lemma \ref{pathwkcont}. This limit, combined with the previous argument, entails that $\lim_{ \varepsilon \rightarrow 0} \sup_{\ee^\prime \in (0,  \varepsilon]}\bE[ \varphi (V_{\ee^\prime \! \!, \,  \ee})] = 0$.    

Therefore, we can 
find $\ee_p \! \in \!(0, 1)$, $p\! \in \!\bbN$, that decreases to $0$ such that $\sum_{p\geq 0}  
\bE [ \varphi( V_{ \ee_{p+1}, \ee_p}) ] \!< \! \infty$, and there exists $\Omega_0\in \ccF$ such that $\bP(\Omega_0)\! =\! 1$ and such that 
$R_p:= \sum_{q \geq p }  \varphi( V_{\ee_{q+1}, \ee_q }) \longrightarrow 0$ as $p\!\rightarrow \!\infty$, on $\Omega_0$.  
We then work determininistically on $\Omega_0$: by the previous arguments, for all Borel subsets $B$ of $[0, x]$, for all $t \!\in \! (s_0, s_0+T)$ and for all 
$q\! >\! p$, we get 
$d( \cZ_t^{\ee_q} (B), \cZ_t^{\ee_p} (B))  \!\leq \!R_p$ and $d( \cZ_{t-}^{\ee_q} (B), \cZ_{t-}^{\ee_p} (B)) 
\! \leq \! R_p$, since $d$ is a distance on $[0, \infty]$. Since $t\!> \! s_0$, the monotone convergence 
for sums entails that $\lim_{q\rightarrow \infty}  \cZ_t^{\ee_q} (B) \!= \!\cZ_t (B)$ and  
$\lim_{q\rightarrow \infty}  \cZ_{t-}^{\ee_q} (B) \!= \!  \cZ_{t-} (B)$. By the continuity of the distance $d$, for all $B$, all $t \! \in\! (s_0, \infty)$ and all $p \!\in\! \bbN$, we get 
$d( \cZ_t (B), \cZ_t^{\ee_p} (B)) \! \leq \! R_p$ and 
$d( \cZ_{t-} (B), \cZ_{t-}^{\ee_p} (B)) \! \leq \!  R_p$. This easily implies that $\cZ$ is  $d_{{\rm var}}$-cadlag on $(s_0, s_0+T)$ since the processes $\cZ^{\ee_p}$ are also 
$d_{{\rm var}}$-cadlag on the same interval. This completes the proof in the infinite variation cases since $s_0$ can be taken arbitrarily small and $T$ arbitrarily large.

\medskip
  
  We next consider the finite variation cases: we fix $s_0 \in (0, \infty)$ and for any $\ee \in (0, 1)$, we set 
$$ \forall t \in [0, s_0] \, , \quad 
\cZ_t^\ee = \sum_{{j\in J}} \un_{\{t_j \leq t\, , \, \rZ^j_{0} >\ee \}} \, \rZ^j_{t-t_j} \, \delta_{x_j} . $$
Since $\#\{ j \!\in\! J : t_j \leq s_0 \, , \, \rZ^j_{0} \!>\! \ee \}$ is a Poisson r.v.~with mean $x\, \pi ((\ee, \infty])\int_0^{s_0} e^{-Dt} dt  \!<\! \infty$, $\cZ^\ee$, as a process indexed by $[0, s_0]$, is a finite sum of weighted Dirac masses whose weights are cadlag $[0, \infty]$-valued processes on $[0, s_0]$: by an easy argument, it is $d_{{\rm var}}$-cadlag on $[0, s_0]$. Next observe that for any 
$ \varepsilon \! >\! \ee^\prime \! > \! 0$,     
$\cZ_t^{\ee^\prime}\!= \!\cZ_t^{\ee}+ \sum_{{j\in J}} \un_{\{t_j \leq t\, , \, \rZ^j_{0}   \in  (\ee^\prime,\ee] \}} \, \rZ^j_{t-t_j} \, \delta_{x_j}$. Thus,  
$$ \sup_{^{t \in [0, s_0]}} \!\! d_{{\rm var}} \big(\cZ_t^{\ee^\prime}, \cZ_t^{\ee} \big) \leq 
\varphi (V_{\ee^\prime \!\!, \,  \ee}) \quad \textrm{where} \quad V_{\ee^\prime \!\!, \,  \ee}:= \sum_{{j\in J}}
\un_{\{t_j \leq s_0\, , \, \rZ^j_{0}  \in  (\ee^\prime,\ee] \}}  \sup_{^{t\in [0, s_0]}}\!\! \rZ^j_t \; .$$
The exponential formula for Poisson point measures then implies for any $\lambda \! \in \! (0, \infty)$,  
$$  -\frac{_1}{^x} \log \bE \big[\exp \big(\!-\! \lambda V_{\ee^\prime \!\!, \,  \ee}  \big) \big] 
=\int_0^{s_0}\!\!\! \!\!\! e^{-Dt}\, dt\!  \int_{(\varepsilon^\prime,\varepsilon]} \!\!\!\!\! \!\!\! \pi(dr)\, 
\bbE_r\Big[1\!- \! e^{-\lambda\sup_{[0,s_0]} \rZ }\Big] \; .$$
We now use Lemma \ref{supZfinvar}: if $D\! \in \! (0, \infty)$, we set $\Ppsi^* (\lambda)= \Ppsi (\lambda)-D\lambda$ and if $D \! \in \! (-\infty, 0]$, we simply take $\Ppsi^*=\Ppsi$. Denote by $u^*$ the function derived from $\Ppsi^*$ as $u$ is derived from $\Ppsi$ by (\ref{integeq}). As a consequence of Lemma \ref{supZfinvar}, we get 
$\bbE_r[1- e^{-\lambda\sup_{[0,s_0]} \rZ }] \leq 1-e^{-ru^*(s_0, \lambda)}$. Thus,  
\begin{eqnarray*}   -\frac{_1}{^x} \log \bE \big[\exp \big(\!-\! \lambda V_{\ee^\prime \!\!, \,  \ee}  \big) \big] &\leq  & \int_0^{s_0}\!\!\! \!\!\! e^{-Dt}\, dt\!  \int_{(\varepsilon^\prime,\varepsilon]} \!\!\!\!\! \!\!\! \pi(dr)\,  
\big( 1\!- \!e^{-ru^*(s_0, \lambda)} \big)\\
& \leq & s_0e^{\lvert D \rvert s_0} u^*(s_0, \lambda) \int_{(0,\varepsilon]} \!\!\!\!\!\! \pi(dr)\, r  \quad   \underset{^{ \ee\rightarrow 0}}{-\!\!-\!\!\!\longrightarrow} \; 0. 
 \end{eqnarray*}
This easily entails $\lim_{ \varepsilon \rightarrow 0} \sup_{\ee^\prime \in (0,  \varepsilon]}\bE[ \varphi (V_{\ee^\prime \! \!, \,  \ee})] = 0$. 
We then argue as in the infinite variation cases: there exists a sequence $\ee_p \! \in \!(0, 1)$, $p\! \in \!\bbN$, that decreases to $0$ and there exists $\Omega_0\in \ccF$ with $\bP(\Omega_0)\! =\! 1$, such that $R_p:= \sum_{q \geq p }  \varphi( V_{\ee_{q+1}, \ee_q}) \longrightarrow 0$ as $p\!\rightarrow \!\infty$, on $\Omega_0$. We work determininistically on $\Omega_0$: we set $\cZ^*_t = \cZ_t -e^{-Dt}\ell (\, \cdot \, \cap [0, x])$, that is the purely atomic part of $\cZ_t$.
Then, for all $B$, for all $t \! \in \! [0, s_0] $ and for all $p \!\in \! \bbN$, 
$d( \cZ^*_t (B), \cZ_t^{\ee_p} (B)) \! \leq \! R_p$ and 
$d( \cZ^*_{t-} (B), \cZ_{t-}^{\ee_p} (B)) \! \leq \! R_p$. This implies 
that $\bP$-a.s.~$\cZ^*$ is $d_{\textrm{var}}$-cadlag on $[0, s_0]$, by the same arguments as in the infinite variation cases. Clearly, a similar result holds true for $\cZ$ on $[0, s_0]$, which 
completes the proof of Proposition \ref{reguZ}, since $s_0$ can be chosen arbitrarily large. \cqfd    
\end{proof}

Note that in the finite variation cases,
$\cZ$ is $d_{{\rm var}}$-right continuous at $0$. In the infinite variation cases, this cannot 
be so: 
indeed, set $B= [0,x] \backslash \{ x_i \, ; \,  i\! \in \! I\}$, then $\cZ_t (B) \!= \! 0$ for any $t \!\in \! (0, \infty)$ 
but $\cZ_0 (B) \!= \! \ell (B) \! =\! x$. However, we have the following lemma. 
\begin{lemma} 
\label{zero} Assume that $\Ppsi$ is of infinite variation type. Let $\cZ$ be defined on $(\Omega, \ccF, \bP)$ by (\ref{cZdef}). Then 
$$ \forall B\in \ccB([0, x]), \quad \textrm{$\bP$-a.s.~$\lim_{t\rightarrow  0+} \cZ_t (B)= \ell(B)$.}  $$
This implies that $\bP$-a.s. $\cZ_t \rightarrow \cZ_0$ weakly as $t\rightarrow 0+$.  
\end{lemma}
\begin{proof} Since $ (\cZ_t (B))_{t\in [0, \infty)}$ has the finite dimensional marginal laws of a CSBP($\Ppsi, \ell (B)$), it admits a modification $Y=(Y_t)_{t\in [0, \infty)}$ that is cadlag on $[0, \infty)$. By Proposition \ref{reguZ}, observe that $\cZ_{\, \cdot \, } (B)$ 
is cadlag on $(0, \infty)$. Therefore, $\bP$-a.s.~$Y$ 
and $\cZ_{\, \cdot \, } (B)$ coincide on $(0, \infty)$, which implies the lemma. \cqfd 
\end{proof}

\subsubsection{Proof of Theorem \ref{construc} and of Theorem \ref{mainth} ($i$).} 
Recall the notation $Z_t = \cZ_t ([0, x])$. By Proposition 
\ref{reguZ}, $Z$ is cadlag on $(0, \infty)$ and by arguing as in Lemma \ref{zero},  
without loss of generality, we can assume that $Z$ is right continuous at time $0$: it is therefore a cadlag CSBP($\Ppsi , x$). Recall from (\ref{absordef}) the definition of the absorption times $\zeta_0$, $\zeta_\infty$ and $\zeta$ of $Z$.      
We first set 
\begin{equation}
\label{Mnonabs}
\forall t \in [0,\zeta), \; \forall B \in \ccB ([0, x]), \qquad M_t (B)= 
\frac{\cZ_t (B)}{Z_t} \; .
\end{equation}
Observe that $M$ has the desired regularity on $[0, \zeta)$ by Proposition \ref{reguZ} and Lemma \ref{zero}. Moreover $M$ satisfies 
property (\ref{Mdef}). It only remains to define $M$ for the times $t \! \geq \! \zeta$ on the event 
$\{ \zeta \! < \!  \infty \}$.

Let us first assume that $\bP (\zeta_0 \!<\!  \infty) \!> \!0$, which can only happen if 
$\Ppsi$ satisfies (\ref{nonpersist}). Note that in this case, $\Ppsi$ is of infinite variation type.  
Now recall $\ccP$ from (\ref{defP}) and $\cZ$ from (\ref{cZdef}). Thus, $\zeta_0\! =\!  \sup_{i\in I} \zeta_0^i$, where $\zeta_0^i$ stands for the extinction time of $\rZ^i$. 
Then, $\bP( \zeta_0  \! < \!  t )= \exp(-x\rN_\Ppsi (\zeta_0 \! \geq \! t ))$. Thus, $  \rN_\Ppsi (\zeta_0 \! \geq \! t)= v(t)$, that is the function defined right after (\ref{nonpersist}) which satisfies  
$\int_{v(t)}^\infty dr/ \Ppsi (r)= t$. Since $v$ is $C^1$, the law (restricted to $(0, \infty)$) 
of the extinction time $\zeta_0$ under $\rN_\Ppsi$ is diffuse. This implies that $\bP$-a.s.~on $\{ \zeta_0 \! < \! \infty\}$ 
there exists a unique $i_0\in I$ such that $\zeta_0= \zeta^{i_0}_0$. 
Then, we set $\xi_0:=\sup \{ \zeta^i_0 \; ; \;  i\in I \backslash \{ i_0\}  \}$, $\ree = x_{i_0}$ and we get 
$M_t= \delta_{\ree}$ for any $t \in (\xi_0, \zeta_0)$. Thus, on the event $\{ \zeta_0 \! < \! \infty\}$ and for any $t\! >\! \zeta_0$, we set $M_t= \delta_{\ree}$ and $M$ has the desired regularity on the event $\{ \zeta_0 \!< \! \infty\}$. An easy argument on Poisson point measures entails that conditional on 
$\{ \zeta_0\! < \! \infty \}$, $\ree$ is uniformly distributed on $[0, x]$.

Let us next assume that $\bP (\zeta_\infty \! <\!  \infty) \! >\! 0$, which can only happen if $\Ppsi$ satisfies (\ref{nonconserv}). We first consider the infinite variation cases: note that 
$\zeta_\infty\! =\!  \inf_{i\in I} \zeta_\infty^i$, where $\zeta_\infty^i$ stands for the explosion time of $\rZ^i$. 
Then, $\bP( \zeta_\infty  \! \geq \!  t )= \exp(-x\rN_\Ppsi (\zeta_\infty \! < \!  t ))$. Thus, $  \rN_\Ppsi (\zeta_\infty  \! < \!  t)= \kappa (t)$ that is the function defined right after (\ref{nonconserv}) which satisfies  
$\int^{\kappa (t)}_0 dr/ (\Ppsi (r))_-= t$. Since $\kappa$ is $C^1$, the law (restricted to $(0, \infty)$) of the explosion time $\zeta_\infty$ under $\rN_\Ppsi$ is diffuse. This implies that $\bP$-a.s.~on $\{ \zeta_\infty \! < \! \infty\}$ 
there exists a unique $i_1\!\in\! I$ such that $\zeta_\infty\!=\! \zeta^{i_1}_\infty$. Then, on $\{ \zeta_\infty \! <\! \infty \}$, we set $\ree= x_{i_1}$ and $M_t = \delta_\ree$, for any $t\! \geq \! \zeta_\infty$. Then, we get 
$ \lim_{t\rightarrow \zeta_\infty -} \lVert  M_t \! -\! \delta_\ree \rVert_{\textrm{var}}\! =\! 0$ and an easy argument on Poisson point measures entails that conditional on $\{ \zeta_\infty \! < \! \infty \}$, $\ree$ is uniformly distributed on $[0, x]$. This completes 
the proof when $\Ppsi$ is of infinite variation type. In the finite variation cases, we argue in the 
same way: namely, by simple computations, one shows that for any $t \! \in \! (0, \infty)$, 
$\#\{j\in J: t_j \leq t, \rZ^j_{t-t_j}\! =\! \infty\}$ is a Poisson r.v.~with mean $x\kappa (t)$; 
it is therefore finite and the times of explosion of the population have diffuse laws:   
this proves that the descendent population of exactly one ancestor explodes strictly before the others, and it implies the desired result in the finite variation cases: the details are left to the reader. \cqfd 
 
\begin{remark}   
\label{afait} 
Note that the above construction of $M$ entails Theorem \ref{mainth} ($i$).  \cq
\end{remark}

\section{Proof of Theorem \ref{mainth}.}
\label{proofsec}

\subsection{Results on Grey martingales.}
\label{Greyresults}

We briefly discuss the limiting laws of Grey martingales (see \cite{Gre74}) associated with CSBP that are involved in describing the asymptotic frequencies of the settlers. Recall from (\ref{cZdef}) and (\ref{cZdef2}) the definition of 
$\cZ_t$: for any $y$ fixed, $t \! \longmapsto \! \cZ_t([0, y])$ is a CSBP($\Ppsi, y$) and for any $t$ fixed, $y\! \longmapsto \! \cZ_t([0, y])$ is a subordinator. Let $\theta \! \in \! (0, \infty)$ and $y\! \in (0, x]$. We assume that $u(-t, \theta)$ is well-defined for any $t\! \in \!(0, \infty)$: namely, we assume that $\kappa (t) \!<\! \theta \!<\! v(t)$, for all $t\! \in \! (0, \infty)$. Recall that (\ref{flow}) extends to negative times. Therefore, 
$t \! \longmapsto \! \exp (-u(-t, \theta) \cZ_t([0, y]) )$ is a $[0, 1]$-valued martingale that a.s.~converges to a limit in $[0, 1]$ denoted by $\exp(-W^\theta _y )$, where $W^\theta_y$ is a $[0, \infty]$-valued random variable. Since $y \! \longmapsto \! u(-t, \theta) \cZ_t([0, y]) $ is a subordinator, 
$y \! \longmapsto \! W^\theta_y$ is a (possibly killed) subordinator. We denote by $\phi_\theta $ its Laplace exponent that has therefore the general L\'evy-Khintchine form: 
$$
 \forall \lambda \in [0, \infty), \quad \phi_\theta (\lambda)= \kappa_\theta + d_\theta \lambda + \int_{(0, \infty)}\!\!\!\! \!\!\! \varrho_\theta (dr) \big( 1-e^{-\lambda r} \big)\, ,
$$
where $\kappa_\theta , d_\theta \! \in \!  [0, \infty)$ and $\int_{(0, \infty)} (1\! \wedge \! r)\,  \varrho_\theta (dr) 
\! <\! \infty$. Note that $\phi_\theta (1)= \theta$, by definition. 
We first consider the behaviour of CSBP when they tend to $\infty$. 
\begin{proposition}
\label{GreyLarge} We assume that $\Ppsi$ is not linear and that $\Ppsi^\prime (0+)\! \in\!  (-\infty, 0)$, which implies that $\Ppsi$ is conservative and $\gamma \! \in \! (0, \infty]$. Let $\theta\! \in \! (0, \gamma)$. Then, $u(-t, \theta)$ is well-defined for all $t \! \in \!  (0, \infty)$ and $\lim_{t\rightarrow \infty} u(-t, \theta)= 0$.  
For any $\theta^\prime \! \in \! (0, \gamma)$ and any $y \! \in\! (0, x]$, we then get $\bP$-a.s.
\begin{equation}\label{ratioW}
 W_y^\theta= R_{\theta^\prime\! \!,  \theta} \,  W^{\theta^\prime}_y \quad \textrm{where} \quad R_{\theta^\prime\! \!, \theta} := \exp \Big( \Ppsi^\prime (0+) 
\int_{\theta^\prime}^{\theta} \!\!  \frac{d\lambda}{\Ppsi (\lambda)} \Big) \; .
 \end{equation}
$W^\theta$ is a conservative subordinator without drift: namely $\kappa_\theta =d_\theta =0$. 
Moreover,  
\begin{equation}\label{phiLarge}  
 \forall \lambda \in (0, \infty), \quad \phi_\theta (\lambda)= u\Big( \frac{{\log \lambda}}{-\Ppsi^\prime (0+)} \,  , \,  \theta \Big) \quad \textrm{and} \quad \varrho_\theta \big( (0, \infty)\big) = \gamma \; .
 \end{equation} 
Thus, if $\gamma <\infty$, $W^\theta$ is a compound Poisson process with jump-rate $\gamma$ and jump-law $\frac{1}{\gamma} \varrho_\theta$ whose Laplace transform is $\lambda\mapsto 1-\frac{1}{\gamma} u( \frac{\log \lambda}{-\Ppsi^\prime (0+)} , \theta)$.  
\end{proposition}
\begin{proof} Let $\theta \! \in \! (0, \gamma)$ and $t \! \in \! (0, \infty)$. Note that $v(t) \! >\! \gamma$ and 
since $\Ppsi$ is conservative, $\kappa (t)\! = \!0$. Thus, for all $t \! \in \! (0, \infty)$, 
$u(-t, \theta)$ is well-defined. 
Note that 
$\Ppsi$ is negative on $(0, \gamma)$, then, by (\ref{integeq}), 
$\lim_{t\rightarrow \infty} u(-t, \theta) \!= \!0$ and 
$\lim_{t\rightarrow \infty} u(t, \theta)\! = \! \gamma$, even if $\gamma \!=\! \infty$. 
Next, observe that 
\begin{equation}\label{ratiou}  
\frac{u(-t, \theta)}{u(-t, \theta^\prime)}  = 
\exp \Big( \int_{\theta^\prime}^{\theta} \!\! d\lambda \, \partial_\lambda \log (u(-t, \lambda)) \Big) 
= \exp \Big( \int_{\theta^\prime}^{\theta} \!\! \frac{\Ppsi (u(-t, \lambda))}{u(-t, \lambda)} \, \frac{d\lambda}{\Ppsi (\lambda)} \Big) \; .
\end{equation}
This entails (\ref{ratioW}) since $\lim_{\lambda \rightarrow 0} \Ppsi (\lambda)/ \lambda= \Ppsi^\prime (0+)$. Thus, $\phi_{\theta} (1/R_{\theta^\prime\! \!, \theta})= \phi_{\theta^\prime} (1)= \theta^\prime$. Then, take $\theta^\prime= u(t, \theta)$: by (\ref{integeq}), it implies that $\phi_\theta (e^{-\Ppsi^\prime (0+) t})= u(t, \theta)$, for any $t\in \bbR$, which proves the formula for $\phi_\theta$ in (\ref{phiLarge}). Next observe that $\kappa_\theta= \lim_{\lambda \rightarrow 0} \phi_\theta (\lambda)= \lim_{t\rightarrow \infty} u(-t, \theta)= 0$. Namely, $W^\theta$ is conservative. Also note that 
$\lim_{\lambda \rightarrow \infty} \phi_\theta (\lambda)= \lim_{t\rightarrow \infty} u(t, \theta)= \gamma$. Thus, if $\gamma <\infty$, $d_\theta= 0$ and the last part of the proposition holds true. 
 
 We next assume that $\gamma = \infty$. Then, $-\Ppsi$ is the Laplace exponent of a conservative subordinator and we are in the finite variation cases.  
We set 
$A(t):= \log (e^{\Ppsi^\prime (0+) t} u(t, \theta))$ and 
we observe that $\log d_\theta= \lim_{t\rightarrow \infty} A(t)$, by taking 
$\lambda = e^{-\Ppsi^\prime (0+) t}$ in (\ref{phiLarge}). An easy comptutation using (\ref{flow}) entails 
\begin{eqnarray*}  A(t)-\log \theta & = &   \int_0^t \!\!  \big( \Ppsi^\prime (0+)  +  \partial_s  \log u(s, \theta) \big) ds  =  
\int_0^t  \!\! \Big( \Ppsi^\prime (0+) \!-\! \frac{\Ppsi (u(s, \theta))}{u(s, \theta)} \, \Big)   ds  \\ 
& = & t \int_0^1  \!\! \Big( \Ppsi^\prime (0+) \!-\! \frac{\Ppsi (u(st, \theta))}{u(st, \theta)} \, \Big)   ds \; .
\end{eqnarray*}
Recall that $\lim_{\lambda \rightarrow \infty} \Ppsi (\lambda)/\lambda= D$. Then, for any $s\in (0, 1]$, 
$$ \lim_{t\rightarrow \infty}  \Ppsi^\prime (0+) \!-\! \frac{\Ppsi (u(st, \theta))}{u(st, \theta)}= 
\Ppsi^\prime (0+) -D= -\int_{(0, \infty)} r \, \pi (dr) < 0 \, , $$ 
since $\pi\neq 0$. This implies that $\lim_{t\rightarrow \infty} A(t)= -\infty$ and thus $d_\theta= 0$. 
\cqfd 
\end{proof}
We complete this result by the following lemma. 
\begin{lemma}\label{asLarge}
We assume that $\Ppsi$ is not linear and that $\Ppsi^\prime (0+)\! \in\!  (-\infty, 0)$, which implies $\Ppsi$ is conservative and $\gamma \! \in \! (0, \infty]$. Let $\theta\! \in \! (0, \gamma)$. Then, $u(-t, \theta)$ is well-defined for all $t \! \in \!  (0, \infty)$ and $\lim_{t\rightarrow \infty} u(-t, \theta)= 0$.  
Moreover, there exists 
a cadlag subordinator $W^\theta$ whose initial value is $0$ and whose Laplace exponent is $\phi_\theta$ as defined by (\ref{phiLarge}) such that 
$$ \textrm{$\bP$-a.s.} \quad \forall y\in [0, x], \quad \lim_{t\rightarrow \infty} u(-t, \theta) \cZ_t ([0, y]) = W^\theta_y \quad \textrm{and} \quad \lim_{t\rightarrow \infty} u(-t, \theta) \cZ_t (\{y\}) = \Delta W^\theta_y\, , $$
where $ \Delta W^\theta_y$ stands for the jump of $W^\theta$ at $y$. 
 \end{lemma} 
\begin{proof} We first assume that $\Ppsi$ is of finite variation type. Fix $ \varepsilon, s_0 \! \in\!  (0, \infty)$. Recall from (\ref{defQ}) the definition of $\ccQ$ and observe that $\sum_{j\in J} \un_{\{ t_j \leq s_0, \rZ^j_0 > \varepsilon \}} \delta_{(x_j, t_j, \rZ^j )}
= \sum_{1\leq n \leq N} \delta_{(X_n , T_n, Z^{(n)})} $, where $N$ is a Poisson r.v.~with mean 
$C:=x D^{-1}(1\!-\! e^{-Ds_0}) \pi (( \varepsilon , \infty))$ and conditionally given $N$, the 
variables $X_n$, $T_n$, $Z^{(n)}$, $1\! \leq\!  n\! \leq\!  N$ are independent: the $X_n$ are uniformly distributed on $[0, x]$, the law of $T_n$ is $(1\!-\!e^{-Ds_0})^{-1} De^{-Dt} \un_{[0, s_0]} (t) \ell (dt)$ and 
the processes $Z^{(n)}$ are distributed as CSBP($\Ppsi$) whose entrance law is $ \pi (( \varepsilon , \infty))^{-1}\un_{( \varepsilon, \infty )} (r) \pi (dr)$. When $D=0$, one should replace $(1-e^{-Dt})D^{-1}$ by $t$ in the last two expressions. 
We next observe that $u(-t, \theta) Z^{(n)}_{t-T_n}= u(-(t-T_n), u(-T_n ,\theta)) Z^{(n)}_{t-T_n} \rightarrow V_n$ exists as $t \rightarrow \infty$  and by Proposition \ref{GreyLarge}, 
\begin{equation}\label{VnLapl}  
 \bE \big[ e^{-\lambda V_n}\big] \!=\! \frac{_1}{^{\pi (( \varepsilon , \infty))}}\!\!\int_{( \varepsilon , \infty)}\! \!\!\!\!\! \!\! \!\! \pi (dr)  \bE \big[e^{-r \phi_{u(-T_n, \theta)}(\lambda)} \big] \!= \! x C^{-1}\int_{0}^{s_0}\!\!\!\! dt \, e^{-Dt}\!\! \int_{( \varepsilon , \infty)} \!\!\! \!\!\!\!\! \!\! \pi (dr)  e^{-r \phi_{u(-t, \theta)}(\lambda)}  
 \end{equation}   
As $ \varepsilon \rightarrow 0$ and $s_0 \rightarrow \infty$, this proves that there exists $\Omega_0\in \ccF$ such that $\bP (\Omega_0)= 1$ and on $\Omega_0$, for any $j\in J$, 
$\lim_{t\rightarrow \infty} u(-t, \theta) \cZ_t (\{ x_j\})=  \lim_{t\rightarrow \infty} u(-t, \theta) \rZ^j_{t-t_j}=: \Delta_j$ exists in $[0, \infty)$. Then, on $\Omega_0$, 
for any $y\in [0, x]$, we set $W^\theta_y= \sum_{j\in J} \un_{[0, y]}(x_j) \Delta_j$ and we take $W^\theta$ as the null process on $\Omega \backslash \Omega_0$. 
Clearly, $W^\theta$ is a cadlag subordinator whose initial value is $0$. We next prove that its Laplace exponent is $\phi_\theta$. To that end fix $y\in (0, x]$; by (\ref{VnLapl})
\begin{eqnarray*}   
\bE \Big[ \exp \Big(\! - \! \lambda \sum_{j\in J}\un_{\{ x_j \leq y \, ; \, \rZ^j_0 > \varepsilon\,  ; \,  t_j \leq s_0\}} \Delta_j \Big)\Big]& = &  \bE \Big[ \exp \Big( \!-\! \lambda\sum_{1\leq n\leq N} \un_{\{ X_n \leq y\}} V_n \Big) \Big]  \\ 
& = &  \exp \Big( \! -\!y \!\! \int_{0}^{s_0}\!\!\!\! dt\, e^{-Dt}\!\! \int_{( \varepsilon , \infty)} \!\!\! \!\!\!\!\! \!\! \pi (dr)  \big( 1-e^{-r \phi_{u(-t, \theta)}(\lambda)} \big) \Big) .
\end{eqnarray*}
Let $ \varepsilon \rightarrow 0$ and $s_0\rightarrow \infty$ to get 
\begin{eqnarray} 
\label{integLarge}
-\frac{_1}{^y}\log \bE\big[ e^{-\lambda W^\theta_y }\big] & = & 
\int_{0}^{\infty}\!\!\!\! \! dt \, e^{-Dt}\!\int_{( 0 , \infty)} \!\!\!\! \!\!\!\!\! \!\! \pi (dr)  \big( 1-
e^{-r \phi_{u(-t, \theta)}(\lambda)} \big) \nonumber\\
&  = & 
 \int_{0}^{\infty}\!\! \!\!\! dt\,  e^{-Dt}\big(  D\phi_{u(-t, \theta)}(\lambda) -\Ppsi (\phi_{u(-t, \theta)}(\lambda) )\big) \; .
 \end{eqnarray}
Then, we set $g(t):=e^{-Dt}\phi_{u(-t, \theta)} (\lambda)$. By (\ref{phiLarge}) and (\ref{flow}), 
$g(t)= e^{-Dt} u(-t, \phi_\theta (\lambda))$. Thus, $\partial_t g (t)= e^{-Dt}(\Ppsi (\phi_{u(-t, \theta)}(\lambda))-D\phi_{u(-t, \theta)}(\lambda))$ and to compute (\ref{integLarge}), we need to specify the limit of $g$ as $t$ tends to $\infty$: since $\lim_{t\rightarrow \infty} u(-t, \phi_\theta (\lambda))=0$, 
$$ \partial_t \log g(t)= \frac{\Ppsi(u(-t, \phi_\theta (\lambda)))}{u(-t, \phi_\theta (\lambda))}-D \underset{t\rightarrow \infty}{-\!\!\!-\!\!\!-\!\!\!\longrightarrow} \Ppsi^\prime (0+)-D= -\int_{(0,\infty)} \!\!\!\!\!\!\! \! \pi(dr) \, r <0 $$
which easily implies that $\lim_{t\rightarrow \infty} g(t) \!=\! 0$ and by (\ref{integLarge}), we get 
$\bE[\exp (-\lambda W^\theta_y)] \!=\! \exp (-y \phi_{\theta} (\lambda))$. Namely, the Laplace exponent of $W^\theta$ is $\phi_\theta$. 

From Proposition \ref{GreyLarge}, for any $y\in [0, x]$, we get $\bP$-a.s.~$\lim_{t\rightarrow \infty}u(-t, \theta)\cZ_t([0,y])=:W^\prime_y$, where the random variable $W^\prime_y$ has the same law as $W^\theta_y$.
Next observe that  
$$u(-t, \theta)\cZ_t([0,y])  =  u(-t, \theta) e^{-Dt} y + \sum_{j\in J} \un_{\{ x_j \leq y \}} u(-t, \theta ) 
\cZ_t (\{ x_j\}) \; .$$ 
Recall from above that $\lim_{t\rightarrow \infty} e^{-Dt} u(-t, \theta)=0$. Thus, by Fatou for sums, we get  
$\bP$-a.s.~$W^\prime_y \geq \sum_{j\in J} \un_{\{ x_j \leq y \} } \Delta_j = W^\theta_y$, which implies 
$W^\prime_y = W^\theta_y$. Then, there exists $\Omega_1 \in \ccF$ such that $\bP(\Omega_1)=1$ and on $\Omega_1$, for any $q \! \in \! \bbQ \cap [0, x]$, $\lim_{t\rightarrow \infty} u(-t, \theta) \cZ_t ([0, q])= W^\theta_q$.

  We next work deterministically on $\Omega_2\! =\! \Omega_0\, \cap \, \Omega_1$. First observe that if 
$y\notin \{ x_j; j \!\in \!J\}$, $\cZ_{t} (\{y\})= 0$. Thus, by definition of $W^\theta$, for any $y\in [0,x]$, we get $\lim_{t\rightarrow \infty} u(-t, \theta) \cZ_t (\{y\}) = \Delta W^\theta_y$. Moreover, for any $y\in [0,x)$ and any $q\in \bbQ \cap [0, x]$ such that  $q>y$, we get 
$$ W^\theta_y \leq \liminf_{t\rightarrow \infty}u(-t, \theta) \cZ_t ([0, y]) \leq \limsup_{t\rightarrow \infty} u(-t, \theta) \cZ_t ([0, y]) \leq W^\theta_q ,$$
the first equality being a consequence of Fatou. Since $W^\theta$ is right continuous, by letting $q$ go to $y$ in the previous inequality we get $\lim_{t\rightarrow \infty}  u(-t, \theta)\cZ_t ([0, y])=W^\theta_y$ for any $y\in [0,x]$ on $\Omega_2$, which completes the proof of the lemma when $\Ppsi$ is of finite variation type. 

\medskip

When $\Ppsi$ is of infinite variation type the proof follows the same lines. Fix $ \varepsilon, s_0 \in (0, \infty)$, recall from (\ref{defP}) the definition of $\ccP$ and recall the Markov property in Theorem \ref{cluster}   (c). Then observe that $\sum_{i\in I} \un_{\{  \rZ^i_{s_0} > \varepsilon \}} \delta_{(x_i, \rZ^i_{s_0 +\,  \cdot\, } )}\! =\! \sum_{1 \leq n\leq N} \delta_{(X_n , Z^{(n)})} $, where $N$ is a Poisson random variable with mean 
$C\!:=\! x\, \rN_\Ppsi (\rZ_{s_0} \!> \!\varepsilon)$ and, conditionally on $N$, the variables $X_n$, 
$Z^{(n)}$, $1 \! \leq \!  n\! \leq \! N$, are independent: $X_n$ is uniformly distributed on $[0, x]$ and the processes $Z^{(n)}$ are CSBP($\Ppsi$) whose entrance law is given by $ \rN_\Ppsi (\rZ_{s_0} \!\in\! dr\,  |\, 
 \rZ_{s_0} \!>\! \varepsilon)$. 
Then, note that $u(-t, \theta) Z^{(n)}_{t-s_0}= u(-(t-s_0), u(-s_0 ,\theta)) Z^{(n)}_{t-s_0} \rightarrow V_n$ exists as $t \rightarrow \infty$ 
 and by Proposition \ref{GreyLarge}, $\bE[ \exp (-\lambda V_n)]= \rN_\Ppsi (\exp ( -\phi_{u(-s_0, \theta)} (\lambda ) \rZ_{s_0}) | \rZ_{s_0} > \varepsilon )$. By letting $ \varepsilon$ and $s_0$ go to $0$, 
this proves that there exists $\Omega_0\in \ccF$ such that $\bP (\Omega_0)= 1$ and on $\Omega_0$, for any $i\in I$, 
$\lim_{t\rightarrow \infty} u(-t, \theta) \cZ_t (\{ x_i\})=  \lim_{t\rightarrow \infty} u(-t, \theta) \rZ^i_{t}=: \Delta_i$ exists in $[0, \infty)$. Then, on $\Omega_0$, 
for any $y\in [0, x]$, we set $W^\theta_y= \sum_{i\in I} \un_{[0, y]}(x_i) \Delta_i$ and we take $W^\theta$ as the null process on $\Omega \backslash \Omega_0$. Clearly, $W^\theta$ is a cadlag subordinator whose initial value is $0$ and we prove that its Laplace exponent is 
$\phi_\theta$ as follows. First note that  
\begin{eqnarray}
\label{infdetai}   
\bE \Big[ \exp \Big(\! - \! \lambda \sum_{i\in I}\un_{\{ x_i \leq y \, ; \, \rZ^i_{s_0} > \varepsilon\, \}} \Delta_i\Big)\Big]& = &  \bE \Big[ \exp \Big( \!-\! \lambda\sum_{1\leq n\leq N} \un_{\{ X_n \leq y\}} V_n \Big) \Big]  \nonumber\\ 
& = &  \exp \Big( \!\! -y\, \rN_\Ppsi \big(\un_{\{ \rZ_{s_0} > \varepsilon \}} \big( 1-e^{-\phi_{u(-s_0, \theta)} (\lambda ) \rZ_{s_0} }\big)\big) \Big) .
\end{eqnarray}
By (\ref{phiLarge}) and (\ref{flow}), we get 
$\rN_\Ppsi \big(\ 1-e^{-\phi_{u(-s_0, \theta)} (\lambda ) \rZ_{s_0} }\big)= \phi_\theta (\lambda)$. Then, by letting $ \varepsilon, s_0 \rightarrow 0$ in (\ref{infdetai}), we get $\bE[\exp (-\lambda W^\theta_y)]= \exp (-y \phi_{\theta} (\lambda))$. We next proceed exactly as in the finite variation cases to complete the proof of the lemma. \cqfd
\end{proof}

We next consider the behaviour of finite variation sub-critical CSBP. 
\begin{proposition}
\label{Greysmall} Let $\Ppsi$ be a branching mechanism of finite variation type such that $\Ppsi^\prime (0+) \! \in \! [0, \infty)$. Then, $\Ppsi$ is conservative and persistent, $D \! \in\!  (0, \infty)$, and 
for all $\theta, t \! \in \!(0, \infty)$, $u(-t, \theta)$ is well-defined and $\lim_{t\rightarrow \infty} u(-t, \theta) \! = \! \infty$. For any $\theta , \theta^\prime \! \in \!(0, \infty)$, and any $y\! \in \! (0, x]$, we also get $\bP$-a.s.  
\begin{equation}\label{ratioww}
 W_y^\theta= S_{\theta^\prime\! \!, \theta} \,  W^{\theta^\prime}_y \quad \textrm{where} \quad S_{\theta^\prime\! \!, \theta} := \exp \Big( D
\int_{\theta^\prime}^{\theta} \!\!  \frac{d\lambda}{\Ppsi (\lambda)} \Big) \; .
 \end{equation}
$W^\theta$ is a conservative subordinator. Namely, $\kappa_\theta= 0$. Moreover, 
\begin{equation}\label{phismall}  
 \forall \lambda \in (0, \infty), \quad \phi_\theta (\lambda)= u\Big( \!-\! \frac{{\log \lambda}}{D} \,  , \,  \theta \Big)  \quad \textrm{and} \quad \varrho_\theta \big( (0, \infty)\big) = \pi \big( (0, \infty)\big) /D \; .
 \end{equation} 
The subordinator $W^\theta$ has a positive drift iff $\int_{(0, 1)} \pi (dr) \, r \! \log 1/r < \infty$. In this case,  
\begin{equation}\label{driftform} 
 \log d_\theta = \log \theta \, -\!  \int_\theta^\infty \!\!  \Big( \!\frac{D}{\Ppsi (\lambda)}\! -\! \frac{1}{\lambda} \Big) \, d\lambda \; . 
\end{equation} 
\end{proposition}
\begin{proof} 
Since $\Ppsi$ is conservative and persistent, $\kappa (t) \!=\! 0$ and $v(t)\! = \! \infty$ and 
$u(-t, \theta)$ is well-defined for any $\theta \! \in \! (0, \infty)$. Moreover, (\ref{integeq}) implies  
$\lim_{t\rightarrow \infty} u(-t, \theta) \! = \! \infty$ and $\lim_{t\rightarrow \infty} u(t, \theta)\! =\! 0$. Recall that 
$\lim_{\lambda \rightarrow \infty} \Ppsi (\lambda) / \lambda= D$. Then, (\ref{ratiou}) entails (\ref{ratioww}). We then argue as in the proof of Proposition \ref{GreyLarge} to prove that $\phi_\theta (e^{-Dt})= u(t, \theta)$ for any $t\in \bbR$, which entails the first part of (\ref{phismall}). Thus, 
$\kappa_\theta = \lim_{\lambda \rightarrow 0} \phi_\theta (\lambda)= \lim_{t\rightarrow \infty} u(t, \theta)=0$
and $W^\theta$ is conservative. 

  We next compute the value of $d_\theta$. To that end, we set $B(t)= \log (e^{-Dt} u(-t, \theta))$ and we observe that $\log d_\theta
= \lim_{t\rightarrow \infty} B(t)$, by taking $\lambda = e^{Dt}$ in (\ref{phismall}). By an easy computation using (\ref{flow}), we get 
\begin{equation}\label{drifttform}  
 \log \theta -\! B(t)\! = \! \!
\int_{-t}^0 \!\! ds \big( D + \partial_s \log u(s, \theta)  \big) \! = \!\! \int_0^t \!\! ds  \Big(D- \frac{\Ppsi (u(-s, \theta))}{u(-s, \theta) } \Big) \!=\!\! \int_{\theta}^{u(-t,\theta)} \!\!\!\! \!\!\!\! \!\!\!\! d \lambda \; \; 
 \Big(\frac{D}{\Ppsi (\lambda)} \! -\! \frac{1}{\lambda}\Big). 
\end{equation}
Now recall that $D-\lambda^{-1}\Ppsi (\lambda)= \int_{(0, \infty)} \pi (dr) (1-e^{-\lambda r} )/ \lambda$. Thus, 
\begin{equation} \label{driftt}
 \log \theta -B(t) =
 \int_{(0, \infty)} \!\! \! \! \! \! \!\!\!\!\! \pi (dr) \!\! \int_\theta^{u(-t, \theta)} \!\!\!\! \!\! \!  \!\!\! d\lambda \;  \frac{\, 1-e^{-\lambda r}}{\lambda \Ppsi (\lambda)} \underset{^{t\rightarrow \infty}}{-\!\!\! \longrightarrow }  \int_{(0, \infty)} \!\! \! \! \! \! \!\!\!\!\!  \pi (dr) \!\! \int_\theta^{\infty} \!\!\!\!  d\lambda \;  \frac{1-e^{-\lambda r}}{\lambda \Ppsi (\lambda)}=:I \; .
\end{equation}
Now observe that $\lambda \mapsto \lambda^{-1}\Ppsi (\lambda)$ is increasing and tends to $D$ as $\lambda \rightarrow \infty$. Thus, $\frac{1}{D}J\leq I \leq \frac{ \theta}{\Ppsi (\theta)} J $ where  
$$ J:= \int_{(0, \infty)} \!\! \! \! \! \! \!\!\!\!\! \pi (dr) \!\! 
\int_\theta^{\infty} \!\!\!\!  d\lambda \;  \frac{1-e^{-\lambda r}}{\lambda^2 } = 
\int_{(0, \infty)} \!\! \! \! \! \! \!\!\!\!\! \pi (dr) \, r \! \int_{\theta r}^{\infty} \!\!\!\!  d\mu \;  \frac{1-e^{-\mu}}{\mu^2 } . $$
Clearly, $J< \infty$ iff $\int_{(0, 1)} \pi (dr) \, r \log 1/r <\infty$, which entails the last point of the proposition. By an easy computation, (\ref{drifttform}) implies (\ref{driftform}).  

  It remains to prove the second equality in (\ref{phismall}). First assume that 
$d_\theta= 0$. 
This implies that $\pi ((0, \infty))= \infty$ and the first part of (\ref{phismall}) entails $\varrho_\theta ((0, \infty))= \lim_{\lambda \rightarrow \infty} \phi_\theta (\lambda)= \lim_{t\rightarrow \infty} u(-t, \theta)= \infty$, which proves the second part of (\ref{phismall}) in this case. We next assume that $d_\theta \!>\! 0$. We set $C(t)= u(-t, \theta) -d_\theta e^{Dt}$. Thus, $\varrho_\theta ((0, \infty))= \lim_{t\rightarrow \infty} C(t)$. By (\ref{drifttform}), we get 
$$ \frac{C(t)}{u(-t, \theta)}=1-\frac{d_\theta}{e^{-Dt} u(-t, \theta)} = 
1-\exp \Big(\! -\!\! \int_{u(-t, \theta)}^\infty \!\!\!\! \!\!\!\! \!\!\!\! d\lambda \,  \big(\frac{_D}{^{\Ppsi (\lambda)}}-\frac{_1}{^{\lambda}} \big) \Big)\,  \sim_{t\rightarrow \infty} \int_{u(-t, \theta)}^\infty  \!\!\!\! \!\!\!\! \!\!\!\! d\lambda \,  
\big(\frac{_D}{^{\Ppsi (\lambda)}}-\frac{_1}{^{\lambda}} \big) . $$
Then, $C(t) \sim_{t\rightarrow \infty} F( u(-t, \theta))$ where $F(x)= x\int_x^\infty \big(\frac{D}{{\Ppsi (\lambda)}}-\frac{1}{{\lambda}}  \big) d\lambda$. We then set $\varphi (\lambda)= 
D\lambda-\Ppsi (\lambda)$ and we observe that $\lim_{\lambda \rightarrow \infty} \varphi (\lambda)= \pi ((0,\infty))$. Thus, 
$$ F(x)= x\int_x^\infty \frac{\varphi (\lambda)}{\lambda \Ppsi (\lambda) } d\lambda= \int_{1}^\infty 
\frac{ \varphi (x\mu) x}{\mu \Ppsi (x\mu) } d\mu\; \underset{^{x\rightarrow \infty}}{-\!\!\! \longrightarrow }\;  \frac{\pi ((0, \infty))}{D} \int_1^\infty \frac{d\mu}{\mu^2}= \frac{\pi ((0, \infty))}{D} , $$
 which implies the second part of (\ref{phismall}). \cqfd 
 \end{proof}

We complete this result by the following lemma. 
\begin{lemma}\label{assmall}
Let $\Ppsi$ be a branching mechanism of finite variation type such that $\Ppsi^\prime (0+) \! \in \! [0, \infty)$. Then, $\Ppsi$ is conservative and persistent, $D \! \in\!  (0, \infty)$, and 
for all $\theta , t\! \in \!(0, \infty)$, $u(-t, \theta)$ is well-defined and $\lim_{t\rightarrow \infty} u(-t, \theta) \! = \! \infty$.
Moreover, there exists 
a cadlag subordinator $W^\theta$ whose initial value is $0$ and whose Laplace exponent is $\phi_\theta$ as defined by (\ref{phismall}) such that 
$$ \textrm{$\bP$-a.s.} \quad \forall y\in [0,x], \quad \lim_{t\rightarrow \infty} u(-t, \theta) \cZ_t ([0, y]) = W^\theta_y \quad \textrm{and} \quad \lim_{t\rightarrow \infty} u(-t, \theta) \cZ_t (\{y\}) = \Delta W^\theta_y\, , $$
where $ \Delta W^\theta_y$ stands for the jump of $W^\theta$ at $y$. 
\end{lemma} 
\begin{proof} The proof Lemma \ref{asLarge} works verbatim, except that in (\ref{integLarge})
$$  \int_{0}^{\infty}\!\! \!\!\! dt\,  e^{-Dt}\big(  D\phi_{u(-t, \theta)}(\lambda) -\Ppsi (\phi_{u(-t, \theta)}(\lambda) )\big)= \phi_\theta (\lambda) -d_\theta \lambda \; , $$
which is easy to prove since $e^{-Dt}\phi_\theta (e^{Dt}\lambda) \rightarrow d_\theta \lambda$ as 
$t \rightarrow \infty$. \cqfd 
\end{proof}

\subsection{Proof of Theorem \ref{mainth} \textrm{(\textit{ii-b})},   
\textrm{(\textit{ii-c})} and \textrm{(\textit{iii-b})}.}
\label{noevesec}
We now consider the cases where there is no Eve property. 
Recall that $x\in (0, \infty)$ is fixed and that $\ell$ stands for Lebesgue measure on $\bbR$ or on 
$[0, x]$ according to the context.  
Recall that $\Ppsi$ is not linear and recall the notation $Z_t:=\cZ_t ([0, x])$. 
We first need the following elementary lemma.  
\begin{lemma}
\label{detcvvar} 
For any $t\! \in\! (0, \infty]$, let $m_t \! \in \! \cM_1 ([0, x])$ be of the form 
$m_t \! =\!  a_t   \ell \! +\!  \sum_{y\in S} m_t (\{y\})\delta_y $,  
where $S$ is a fixed countable subset of $[0, x]$ and $a_t \!  \in \!  [0, \infty)$. We assume that for any 
$y \! \in\!  S$, 
$\lim_{t\rightarrow \infty} m_t(\{y\}) \! = \! m_\infty (\{ y\}) $ and 
$\lim_{t\rightarrow \infty} a_t \!= \!a_\infty$. Then, 
$ \lim_{t\rightarrow \infty} \lVert m_t \!-\! m_\infty \rVert_{\mathrm{var}}\! =\! 0$. 
\end{lemma} 
\begin{proof} For all $ \varepsilon \! \in \!(0, \infty)$, there is $S_\varepsilon \! \subset \! S$, finite and 
such that $ \sum_{y\in S\backslash S_\varepsilon} m_\infty (\{ y\}) \! < \! \varepsilon$. Then, for any 
$A \! \subset \! [0, x]$
\begin{eqnarray*}\lvert m_t (A)\!-\!m_\infty (A)\rvert  \!\!\!  \! & \leq  & \!\!\!   
x \, \lvert a_t \!-\!a_\infty \rvert  + \!  \sum_{y\in S_\varepsilon}  \lvert m_t (\{y\}) \!-\! m_\infty (\{ y\}) \rvert  + \! \! \sum_{y\in S\backslash S_\varepsilon}  \!\! m_t  (\{ y\}) + \!\!  \sum_{y\in S\backslash S_\varepsilon}  \!\! 
m_\infty (\{ y\})  \\ 
\!\!\!\!\! &\leq  & \!\!\!  x\,  \lvert a_t \! -\! a_\infty \rvert  + \!\! \sum_{y\in S_\varepsilon}   \lvert m_t (\{ y\}) \! -\! 
m_\infty (\{ y\}) \rvert  +1\!-\! a_t x - \!\! \sum_{y\in S_\varepsilon} \!\!  m_t (\{ y\}) +
\varepsilon .
 \end{eqnarray*}
Thus, 
$$  \limsup_{t\rightarrow \infty} \sup_{A\subset [0, x]} \lvert m_t (A)\!-\!m_\infty (A)\rvert \leq 1\!-\! a_\infty x - \!\! \sum_{y\in S_\varepsilon} \!\!  m_\infty (\{ y\}) +
\varepsilon = \varepsilon+  \!\!  \sum_{y\in S\backslash S_\varepsilon}  \!\! m_\infty (\{ y\}) \leq 2 \varepsilon \, , $$
which implies the desired result. \cqfd \end{proof}

\paragraph{Proof of Theorem \ref{mainth} (\textit{ii-b}) and (\textit{ii-c}).} 
Recall that $B\!=\! \{ \zeta \!=\! \infty \, ; \, \lim_{t\rightarrow \infty} Z_t \!=\! \infty\}$. We assume that $\Ppsi^\prime (0+)\! \in \! (-\infty, 0)$, which implies $\gamma \!\in \!(0, \infty]$ and that $\Ppsi$ is conservative. 
Let $\theta \in (0, \gamma)$ and let $W^\theta$ be a cadlag subordinator as in Lemma \ref{asLarge}. Recall that its Laplace exponent 
is $\phi_\theta$ as defined by (\ref{phiLarge}). 
It is easy to prove that $\bP$-a.s.~$\un_{\{ W^\theta_x>0\}} \!=\! \un_{B}$. 
We now work a.s.~on $B$: it makes sense to set $M_\infty (dr)= dW^\theta_r /W^\theta_x$ that does not depend on $\theta$ as proved by (\ref{ratioW}) in Proposition \ref{GreyLarge}. Note that $M_t= a_t \ell + \sum_{y\in S} M_t (\{ y\}) \delta_y$ either with $a_t \!=\! 0$ and $S=\{ x_i\, ;\,  i\! \in \! I\}$ if $\Ppsi$ is of infinite variation type, or with $a_t \!=\! e^{-Dt}/Z_t $ and $S=\{ x_j\, ;\,  j\! \in \! J\}$ if $\Ppsi$ is of finite variation type. Next note that 
$\{ y \! \in\!  [0, x]: \Delta W^\theta_y \!>\! 0\} \subset S$ and since $W^\theta$ has no drift, we get 
$M_\infty \!=\! \sum_{y\in S} M_\infty (\{ y\}) \, \delta_y $. Then, 
Lemma \ref{asLarge} easily entails that a.s.~on $B$, for any $y\!\in \! S$, $\lim_{t\rightarrow \infty}M_t (\{ y\})\! =\!  M_\infty (\{ y\})$. Next, recall from the proof of Proposition \ref{GreyLarge} that $\lim_{t\rightarrow \infty} u(-t,\theta) e^{-Dt} \!=\!  d_\theta \!=\!  0$, which implies that $\lim_{t\rightarrow \infty} a_t \! = \! 0$. Then, Lemma 
\ref{detcvvar} entails that a.s.~on $B$, $\lim_{t\rightarrow \infty}
\lVert M_t -M_\infty \rVert_{\textrm{var}}= 0$. 

If $\gamma < \infty$, then Proposition \ref{GreyLarge} entails that $W^\theta$ is a compound Poisson process: in this case and on $B$, there are finitely many settlers and conditionally on $B$, the number of settlers is distributed as a Poisson r.v.~with parameter $x\gamma$ conditionned to be non zero, which completes the proof of Theorem \ref{mainth} (\textit{ii-b}). 
If $\gamma \!= \! \infty$, then the same proposition shows that $W^\theta$ has a dense set of jumps. Therefore, a.s.~on $B$ there are a dense countable set of settlers, which completes the proof of Theorem \ref{mainth} (\textit{ii-c}). In both cases, 
the asymptotic frequencies are described by Proposition \ref{GreyLarge} and Lemma \ref{asLarge} \cqfd

\paragraph{Proof of Theorem \ref{mainth} (\textit{iii-b}).} Recall that 
$C\!=\!\{ \zeta \!=\! \infty \, ; \, \lim_{t\rightarrow \infty} Z_t\!=\! 0\}$. We assume that $\Ppsi$ is of finite variation type, which implies that $\Ppsi$ is persistent. Also recall that 
$\bP(C)\!=\! e^{-\gamma x} \! >\!0$. Thus, we also assume that $\gamma \!< 
\! \infty$.  
Then, observe that $\cZ$ under $\bP ( \, \cdot \, | \, C)$ is distributed as the process derived from the 
finite variation sub-critical branching mechanism $\Ppsi (\cdot +\gamma)$. So, 
without loss of generality, we can assume that $\Ppsi$ is of finite variation 
and sub-critical, namely $\Ppsi^\prime (0+) \!\in\! [0,\infty)$, which implies that $\Ppsi$ is conservative and $D \!\in \! (0, \infty)$.  

Let $\theta \! \in \! (0, \infty)$ and let $W^\theta$ be a cadlag subordinator  
as in Lemma \ref{assmall} whose Laplace exponent $\phi_\theta$ is defined by (\ref{phismall}). 
Since $\Ppsi$ is conservative and persistent, it makes sense to set $M_\infty (dr)\! =\! dW^\theta_r /W^\theta_x$ that does not depend on $\theta$ as proved by (\ref{ratioww}) in Proposition \ref{Greysmall}. Note that $M_t \! = \! a_t \ell + \sum_{y\in S} M_t (\{ y\})$ where $a_t \!=\!  e^{-Dt}/Z_t $ and 
$S\!=\!\{ x_j \, ;\,  j\in J\}$, and observe that $\{ y\!\in \! [0, x]: \Delta W^\theta_y \!>\!0\} \subset S$. Recall that $d_\theta$ stands for the (possibly null) drift of $W^\theta$. Then, we get $M_\infty = a_\infty \ell +\sum_{y\in S} M_\infty (\{ y\})\,  \delta_y $, where 
$a_\infty \!=\! d_\theta/W^\theta_x$. 
By Lemma \ref{assmall}, a.s.~for any $y\!\in\! S$, $\lim_{t\rightarrow \infty}M_t (\{ y\})\!=\! M_\infty (\{ y\})$ and recall from the proof of Proposition \ref{GreyLarge} that $\lim_{t\rightarrow \infty} u(-t,\theta) e^{-Dt}\!= \!d_\theta$, which implies that $\lim_{t\rightarrow \infty} a_t= a_\infty$. 
Then, Lemma \ref{detcvvar} entails that a.s.~$\lim_{t\rightarrow \infty}
\lVert M_t -M_\infty \rVert_{\textrm{var}}= 0$. 

If $\pi((0,1)) \! < \! \infty$, then $\pi ((0, \infty)) \! < \!  \infty$ and $\int_{(0, 1)} \pi (dr) \, 
r \log 1/r \!< \!\infty$. 
Proposition \ref{Greysmall} entails that $W^\theta$ has a drift part and finitely many jumps in $[0, x]$:  there is dust and finitely many settlers. More precisely, conditionally given $C$, the number of settlers is distributed as a Poisson r.v.~with parameter $\frac{x}{D} \int_{(0, \infty)} e^{-\gamma r} \pi (dr)$ since 
$e^{-\gamma r}\pi(dr)$ is the L\'evy measure of $\Ppsi( \cdot + \gamma)$. This proves Theorem \ref{mainth} (\textit{iii-b1}).  
If $\pi((0, 1)) \!= \! \infty$ and $\int_{(0, 1)} \pi (dr) \, 
r \log 1/r \!< \! \infty$, Proposition \ref{Greysmall} entails that $W^\theta$ has a drift part and a dense set of 
jumps in $[0, x]$: thus, a.s.~on $C$, there is dust and infinitely many settlers. 
This proves Theorem \ref{mainth} (\textit{iii-b2}).  
Similarly, if $\int_{(0, 1)} \pi (dr) \, 
r \log 1/r \!=\! \infty$, Proposition \ref{Greysmall} entails that a.s.~on $C$, there is no dust 
and there are infinitely many settlers, which proves Theorem \ref{mainth} (\textit{iii-b3}).  
In all cases, conditionally on $C$, 
the asymptotic frequencies are described thanks to Proposition \ref{Greysmall} and Lemma \ref{assmall}
applied to the branching mechanism $\Ppsi(\cdot + \gamma)$.  \cqfd

\subsection{Proof of Theorem \ref{mainth} \textrm{(\textit{ii-a})}  and \textrm{(\textit{iii-a})}.} 
\label{infineve}
\subsubsection{Preliminary lemmas.}
Recall that $x\in (0, \infty)$ is fixed and recall that $\ccM_1([0, x])$ stands for the set of Borel probability measures on $[0, x]$. 
We first recall (without proof) 
the following result -- quite standard -- on weak convergence in $\ccM_1([0, x])$.   
\begin{lemma} \label{cvweak}
For any $t\in [0, \infty)$, let $m_t \! \in \! \ccM_1 ([0, x])$ be such that for all $q\! \in\!  \bbQ\, \cap \, [0, x]$, $\lim_{t\rightarrow \infty} m_t ([0, q])$ exists. Then, there exists $m_\infty \! \in \! 
\ccM_1([0, x])$ such that 
$\lim_{t\rightarrow \infty} m_t  \! = \! m_\infty$ with respect to 
the topology of the weak convergence. 
\end{lemma} 
Recall the definition of $(M_t)_{t\in [0, \infty)}$ from Theorem \ref{construc} and Section \ref{constrMsec}. 
\begin{lemma} \label{Mcvweak}
We assume that $\Ppsi$ is not linear and conservative. Then, there exists a random probability measure $M_\infty$ on $[0, x]$ such that $\bP$-a.s.~$\lim_{t\rightarrow \infty} M_t  \! = \! M_\infty$ with respect to 
the topology of the weak convergence. 
\end{lemma} 
\begin{proof} By Lemma \ref{cvweak}, it is sufficient to prove that for any 
$q\! \in\!  \bbQ\, \cap \, [0, x]$, $\bP$-a.s.~$\lim_{t\rightarrow \infty} M_t ([0, q])$ exists. To that end, we use 
a martingale argument: for any $t\in [0, \infty)$, we denote by $\ccG_t$ the sigma-field generated by 
the r.v.~$\cZ_s([0, q])$ and $\cZ_s((q, x])$, where $s$ ranges in $[0, t]$. Recall that 
$Z_t= \cZ_t([0, q])+\cZ_t((q, x])$ and that $(\cZ_t([0, q]))_{t\in [0, \infty)}$ and 
$(\cZ_t((q, x]))_{t\in [0, \infty)}$ are two independent conservative CSBP($\Ppsi$). Then, for any $\lambda, \mu \in (0, \infty)$ and any $t,s\in [0, \infty)$ 
$$ \bE \big[ \exp \big(\! -\! \mu \cZ_{t+s}([0, q]) \!-\! \lambda Z_{t+s}  \big)\,  \big| \, \ccG_t \big]= \exp \big(\!-\! u(s, \lambda \!+\! \mu)\cZ_{t}([0, q]) \!- \!  u(s, \lambda)\cZ_t((q, x])  \big) $$
By differentiating in $\mu=0$, we get 
\begin{equation} \label{cdexpfull}
\bE[  \un_{\{ Z_{t+s} >0\}} \cZ_{t+s}([0, q]) \, e^{-\lambda Z_{t+s} } \, |\,  \ccG_t]=  \un_{\{ Z_{t} >0\}} 
\cZ_{t}([0, q])  \, e^{-u(s, \lambda) Z_t } \,  \partial_\lambda u \, (s, \lambda) \; .
 \end{equation}
By continuity in $\lambda$, (\ref{cdexpfull}) 
holds true $\bP$-a.s.~for all $\lambda\in [0, \infty)$. We integrate (\ref{cdexpfull}) 
in $\lambda$: note that for any $z\in (0, \infty)$, 
$I(z)\! :=\! \int_0^\infty \! d\lambda \,  e^{-u(s, \lambda) z}\,  \partial_\lambda u \, (s, \lambda) \!=\!  
z^{-1} (1\!- \! e^{-v(s) z})$ if $\Ppsi$ is non-persistent (here $v$ is the function defined right after (\ref{nonpersist})) and $I(z)\!=\! z^{-1}$ if $\Ppsi$ is persistent. In both cases, $I(z) \!\leq \!z^{-1}$ and thus we get 
$$ \bE[\un_{\{  Z_{t+s}>0\} } M_{t+s} ([0, q])  \, | \, \ccG_t ]= 
\bE[  \un_{\{ Z_{t+s} >0\}} \frac{_{\cZ_{t+s}([0, q])}}{^{Z_{t+s}} } \, |\,  \ccG_t] \leq 
\un_{\{ Z_{t} >0\}}  \frac{_{\cZ_{t}([0, q])}}{^{Z_{t}} }= \un_{\{  Z_{t} >0\} } M_{t} ([0, q]) .  $$
Then, $t\longmapsto \un_{\{  Z_{t} >0\} } M_{t} ([0, q]) $ is a nonnegative super-martingale: it almost surely converges and Lemma \ref{cvweak} applies on the event $\{ \zeta_0 \!= \! \infty\}$. Since 
we already proved that $M$ has an Eve on the event $\{ \zeta_0 \! < \! \infty\}$, the proof is complete.  \cqfd
\end{proof}
For any $v \in [0, 1)$ and any $t\in (0, \infty]$, we set 
\begin{equation}
\label{invR}
 R^{-1}_t (v) = \inf \big\{ y\in [0, x] \, : \, M_t ([0, y]) > v \, \big\} \; .
\end{equation} 
Let $U, V: \Omega \rightarrow [0, 1)$ be two independent 
uniform r.v.~that are also independent of the Poisson point measures $\ccP$ and $\ccQ$. 
Then, for any $t, s \! \in \! (0, \infty]$, the conditional law of 
$(R^{-1}_t (U), R_s^{-1} (V))$ given $\ccP$ and $\ccQ$ is 
$M_t \!  \otimes \! M_s$. Moreover, Lemma \ref{Mcvweak} and standard arguments entail 
\begin{equation}
\label{asUV}
\textrm{$\bP$-a.s.} \qquad \lim_{t\rightarrow \infty} R^{-1}_t (U)=R^{-1}_\infty (U) \quad \textrm{and} \quad 
\lim_{t\rightarrow \infty} R^{-1}_t (V)=R^{-1}_\infty (V)  . 
\end{equation}
For any $t\! \in \! (0, \infty)$, we recall the definition of the function $v(t)= \lim_{\lambda \rightarrow \infty} u(t, \lambda)$ that is
infinite if $\Ppsi$ is persistent and finite if $\Ppsi$ is non-persistent. 
Recall that $u(-t, \cdot): (\kappa (t), v(t)) \rightarrow (0, \infty)$ is the reciprocal function of $u(t, \cdot)$. It is increasing and one-to-one, which implies that $\lim_{\lambda \rightarrow v(t)} u(-t, \lambda)= \infty$.   
\begin{lemma}
\label{keyL} Let us assume that $\Ppsi$ is conservative. 
Then, for all $t, \theta \!\in \! (0, \infty)$ and all $s \! \in\! [0, \infty)$, 
\begin{equation} 
\label{kieq}
\bE \big[  \un_{\{ R^{-1}_t (U) \neq  R_{t+s}^{-1} (V) \} } \big(1\! -\! e^{-\theta Z_{t+s}} \big) \big]
\! =\!  x^2 \!\! \int_0^{v(t)} \!\!\! \!\!\! dw \,  \Ppsi (w) \, e^{-xw} \, 
 \frac{u(\!-t, w)\!-\! \big( u(\!-t, w) \!-\!u(s, \theta)\big)_+}{\Ppsi (u(\!-t, w))}  ,
\end{equation}
where $(\, \cdot \, )_+$ stands for the positive part function. 
\end{lemma}
\begin{proof} 
We first prove the lemma when $\Ppsi$ is of infinite variation type. Recall from (\ref{defP})
the definition of $\ccP$ and observe that the jumps of the distribution function of the random measure $M_t$ are given by the collection $\rZ^i_{t}/\sum_{k\in I} \rZ^k_{t}, i\in I$. Recall that $R^{-1}_t$ stands for the inverse of this distribution function. If $i\in I$ (resp.~$j\in I$) is the index of the jump in which $R^{-1}_t (U)$ (resp.~$R_{t+s}^{-1} (V)$) falls then $\{R^{-1}_t (U) \neq  R_{t+s}^{-1} (V) \}$ is the event where $i$ and $j$ are distinct. Consequently on the event $\{ Z_{t+s} \!> \!0\}$ we get,
$$ \bE \big[ \un_{\{   R^{-1}_t (U) \neq  R_{t+s}^{-1} (V) \}} \, \big| \, \ccP \, \big]= 
\sum_{\substack{i , j\in I\\i\neq j}} 
\frac{\rZ^i_t  \; \rZ^j_{t+s}}{ \big(\rZ^i_{t} \!+\!  \rZ^j_{t} \!+\! \sum_{k\in I\backslash\{ i,j\}} \rZ^k_{t} \big)\big(\rZ^i_{t+s} \!+\!  \rZ^j_{t+s} \!+\! \sum_{k\in I\backslash\{ i,j\}} \rZ^k_{t+s} \big) } . $$
Hence,
$$\bE \big[ \un_{\{   R^{-1}_t (U) \neq  R_{t+s}^{-1} (V) \}} \big(1\! -\! e^{-\theta Z_{t+s}} \big)  \big| \ccP \big]
\!=\!\! \sum_{\substack{i , j\in I\\i\neq j}}\! 
\frac{\rZ^i_t  \; \rZ^j_{t+s}  \Big(1\! -\! e^{-\theta \big( \rZ^i_{t+s} \!+\!  \rZ^j_{t+s} \!+\! \sum_{k\in I\backslash\{ i,j\}} \rZ^k_{t+s} \big)} \Big)   }{ \big(\rZ^i_{t} \!+\!  \rZ^j_{t} \!+\! \sum_{k\in I\backslash\{ i,j\}}\! \rZ^k_{t} \big)\big(\rZ^i_{t+s} \!+\!  \rZ^j_{t+s} \!+\! \sum_{k\in I\backslash\{ i,j\}}\! \rZ^k_{t+s} \big) } . $$
%
To simplify notation, we denote by $A$ the left member in (\ref{kieq}). By applying formula (\ref{rePalm}), we get 
$$A =  x^2 \!\!  \int \! \rN_\Ppsi (d\rZ) \!\!  \int \!  \rN_\Ppsi (d\rZ^\prime) \;  \bE \! \left[ 
\frac{\un_{\{ \rZ_{t+s} + \rZ^\prime_{t+s}  + Z_{t+s}>0 \}}\rZ_t  \; \rZ^\prime_{t+s} \big(1-e^{-\theta ( \rZ_{t+s} + \rZ^\prime_{t+s}  + Z_{t+s})} \big)}{ ( \rZ_t + \rZ^\prime_t  + Z_t )( \rZ_{t+s} + \rZ^\prime_{t+s}  + Z_{t+s})}\right] .$$
For any $\lambda_1, \lambda_2 \in (0, \infty)$, we then set 
\begin{eqnarray*} 
B( \lambda_1, \lambda_2) &= & \int \! \rN_\Ppsi (d\rZ) \!\! 
  \int \! \rN_\Ppsi (d\rZ^\prime) \; \bE \! \left[ \rZ_t \rZ^\prime_{t+s} e^{-\lambda_1 (\rZ_t + \rZ^\prime_t  + Z_t)} e^{-\lambda_2 (\rZ_{t+s} + \rZ^\prime_{t+s}  + Z_{t+s})} \right] \\
& =&  \rN_\Ppsi 
\big(\rZ_t e^{-\lambda_1 \rZ_t -\lambda_2 \rZ_{t+s}} \big) \, \rN_\Ppsi 
\big(\rZ_{t+s} e^{-\lambda_1 \rZ_t -\lambda_2 \rZ_{t+s}} \big) \, \bE \big[ e^{-\lambda_1 Z_t -\lambda_2Z_{t+s}} \big]  
\end{eqnarray*}
Recall that $\rN_\Ppsi (1-e^{-\lambda \rZ_t})= u(t, \lambda)$ and recall Theorem \ref{cluster} (c). 
Then, we first get 
$$ \rN_\Ppsi 
\big(\rZ_t e^{-\lambda_1 \rZ_t -\lambda_2 \rZ_{t+s}} \big) = \rN_\Ppsi 
\big(\rZ_t e^{-(\lambda_1+u(s, \lambda_2))  \rZ_t } \big)= \partial_\lambda u \, ( t, \lambda_1\!+\!  u(s, \lambda_2)). $$ 
By the same argument we get 
\begin{eqnarray*} 
 \rN_\Ppsi 
\big(\rZ_{t+s} e^{-\lambda_1 \rZ_t -\lambda_2 \rZ_{t+s}} \big) &= & \partial_{\lambda_2}  \rN_\Ppsi 
\big(1-e^{-\lambda_1 \rZ_t -\lambda_2 \rZ_{t+s}} \big)=\partial_{\lambda_2}  \rN_\Ppsi 
\big(1-e^{-(\lambda_1+ u(s, \lambda_2) ) \rZ_t } \big) \\
&= & \partial_\lambda u(s, \lambda_2) \, 
 \partial_\lambda u( t, \lambda_1\!+\! u(s, \lambda_2)). 
\end{eqnarray*} 
This implies that 
\begin{equation} \label{Blamb}
B( \lambda_1, \lambda_2) =\partial_\lambda u(s, \lambda_2) \, 
\big(\partial_\lambda u( t, \lambda_1\!+\! u(s, \lambda_2))\, \big)^2 \, 
e^{-xu(t , \lambda_1\!+ u(s, \lambda_2))} \; . 
 \end{equation}
An easy argument then entails that 
$$ A=   x^2 \! \!  \int_0^\infty \!\!\! \! \int_0^\infty \!\!\! d \lambda_1 d \lambda_2 \, \big( B(\lambda_1, \lambda_2) \!- \! B(\lambda_1, \lambda_2 \!+\! \theta )\big) \; .$$
Set $C(\theta):= 
x^2  \int_0^\infty  \int_0^\infty B(\lambda_1, \lambda_2 \!+\! \theta )\,  d \lambda_1 d \lambda_2$. The previous equality shows that $A= C(0)-C(\theta)$. We recall that 
$v(s)= \lim_{\lambda \rightarrow \infty} u(s, \lambda)$ 
and let us compute $C(\theta)$.  
To that end we use the changes of variable $y= u(s, \lambda_2 \!+ \! \theta)$ and $\lambda = \lambda_1+ y$ to get  
\begin{eqnarray*} 
C(\theta) & = & 
x^2 \! \!  \int_0^\infty  \!\!\!\!   d\lambda_1   \!\!  \int_{u(s, \theta)}^{v(s)} \!\!\! \!\!\!\! \! 
d y\;  \big(\partial_\lambda u(\,  t \, , \,  \lambda_1\!+\! y \,) \big)^2 \, e^{-xu(t, \lambda_1+ y)}.\\
& =& x^2 \! \!  \int_0^\infty \!\! d\lambda_1\!\!  \int_{\lambda_1+ u(s, \theta )}^{\lambda_1 +v(s)} \!\!\! \!\!\! \!\!\!  d \lambda \,    
\big(\partial_\lambda u( \, t\, , \,  \lambda)\, \big)^2 \, e^{-xu(t, \lambda )} .
\end{eqnarray*}
Recall from (\ref{integeq}) that $\partial_\lambda u(t, \lambda)= \Ppsi (u(t, \lambda))/ \Ppsi (\lambda)$ and note that $\Ppsi (\lambda)= \Ppsi (u(-t, u(t, \lambda)))$.  
Then, by the change of variable $w= u(t, \lambda)$, we get 
$$ C(\theta) =  x^2 \! \!  \int_0^\infty \!\! d\lambda_1 
\int_{u(t, \lambda_1+ u(s, \theta))}^{u(t, \lambda_1+ v(s))} \!\!\! \!\!\!  \!\!\! \!\! dw  \,    \frac{\Ppsi (w)}{\Ppsi (u(-t, w))} \,  e^{-xw}\; .
 $$
Thus, 
\begin{eqnarray*} 
A &= & C(0)-C(\theta) =  
x^2 \! \!  \int_0^\infty \!\! d\lambda_1 \int_{u(t, \lambda_1)}^{u(t, \lambda_1+ u(s, \theta))}  \!\!\! \!\!\!  \!\!\! \!\!\!   dw \; \;  \frac{\Ppsi (w)}{\Ppsi (u(-t, w))} 
e^{-xw}  \\
& =&x^2 \! \!  \int_0^{v(t)} \!\! dw \!\int_0^\infty \!\! d\lambda_1 \, \un_{\{ u(t, \lambda_1) \leq w  
\leq  u(t, \lambda_1+ u(s, \theta))\}} \, \frac{\Ppsi (w)}{\Ppsi (u(-t, w))} 
e^{-xw} \\
& =& x^2 \int_0^{v(t)} \!\!\! dw \,  \Ppsi (w) \, e^{-xw} \, 
 \frac{u(\!-t, w)\!-\! \big( u(\!-t, w) \!-\!u(s, \theta)\big)_+}{\Ppsi (u(\!-t, w))}  ,
\end{eqnarray*}
which is the desired result in the infinite variation cases. 

 The proof in the finite variation cases is similar except that $\cZ$ and $M$ are derived from the Poisson point measure $\ccQ$ defined by (\ref{defQ}). Note that $\Ppsi$ is persistent. We moreover assume it to be conservative: thus, $Z_t\! \in\! (0, \infty)$, for any $t\in [0, \infty)$.   
 Let $A$ stand for the left member in (\ref{kieq}). Then, $A=A_1+A_2$ where
\begin{equation*}
A_1\!:=\! \bE\Big[\frac{_1}{^{Z_t  Z_{t+s}}} \sum_{^{j \in J}}\un_{\{t_j \leq t\}}\rZ^j_{t-t_j}(Z_{t+s} \!-\! \un_{\{t_j \leq t+s\}}\rZ^j_{t+s-t_j})(1\!-\!e^{-\theta Z_{t+s}}) \Big]
  , \; A_2 \!:=\! \bE \Big[\frac{_{xe^{-Dt}}}{^{Z_t}}(1\!-\!e^{-\theta Z_{t+s}})\Big]. 
\end{equation*}
$A_1$ corresponds to the event where $U$ falls on a jump of $R_t$, while $A_2$ deals with the event where it falls on the dust. The latter gives
$$ A_2 =  xe^{-Dt} \!\! \int_0^\infty \!\!\! \!\! d\lambda \, \bE \big[ e^{-\lambda Z_t} \!-\! e^{-\lambda Z_t-\theta Z_{s+t} } \big]  =   xe^{-Dt} \!\! \int_0^\infty \!\!\! \!\! d\lambda \, \big(e^{-x\, u(t,\lambda)}-e^{-x\, u(t,\, \lambda+u(s,\theta))}\big) .$$
We next observe that $A_1\!=\! \int_0^\infty \! \int_0^\infty d\lambda_1 d\lambda_2 \big(\tilde{B}(\lambda_1,\lambda_2)-\tilde{B}(\lambda_1,\lambda_2+\theta)\big)$, where for any $\lambda_1,\lambda_2 \!\in \!(0,\infty)$ we have set
\begin{eqnarray} 
\label{Btilde}
\tilde{B}(\lambda_1,\lambda_2) \!\!\!&=&\!\!\! \bE\Big[ e^{-\lambda_1 Z_t-\lambda_2 Z_{t+s}} \!
 \sum_{^{j \in J}}
\un_{\{t_j \leq t\}}\rZ^j_{t-t_j} \Big( xe^{-D(t+s)}+\!\!  \sum_{^{k\in J\backslash \{ j\} } }\! \! \un_{\{t_k \leq t+s\}}\rZ^k_{t+s-t_k} \Big) \Big] \nonumber \\
\!\!\!&=&\!\!\!  \bE\big[Z_{t+s}e^{-\lambda_1 Z_t}e^{-\lambda_2 Z_{t+s}}\big] \, 
x \! \! \int_0^t \!\! e^{-Db} db \!  \int_{(0,\infty)}\!\!\!\!\!\!\!\! \pi(dr)\, \bbE_r\big[ \rZ_{t-b}e^{-\lambda_1 \rZ_{t-b}}e^{-\lambda_2 \rZ_{t+s-b}} \big]. 
\end{eqnarray}
Here we apply Palm formula to derive the second line from the first one. 
The first expectation in (\ref{Btilde}) yields
$$\bE\big[Z_{t+s}e^{-\lambda_1 Z_t}e^{-\lambda_2 Z_{t+s}}\big] = \partial_{\lambda}u(s,\lambda_2)\partial_{\lambda}u(t,\lambda_1+u(s,\lambda_2)) \, x \, e^{-xu(t,\lambda_1+u(s,\lambda_2))} .$$
The second term of the product in (\ref{Btilde}) gives 
\begin{eqnarray*}
&& x\! \int_0^t \!\! e^{-Db}  db \! \int_{(0,\infty)}\!\!\! \!\!\!\! \!\!\! \pi(dr)\, \partial_{\lambda}u \, (t\!-\!b,\lambda_1\!+\!u(s,\lambda_2))  \, re^{-r u(t-b,\lambda_1+u(s,\lambda_2))}\\
&=& x\! \int_0^t \!\! e^{-Db}  db \,  \partial_{\lambda} u \, (t\!-\!b,\lambda_1\!+\! u(s,\lambda_2)) 
\big(D\!-\!\Ppsi^\prime \big( u(t\!-\!b,\lambda_1\!+\!u(s,\lambda_2))\big)\big)\\
&=& x\Big(\partial_{\lambda  } u \, (t,\lambda_1\!+\!u(s,\lambda_2))-e^{-Dt}\Big)
\end{eqnarray*}
Here, to derive the second line from the first one, we use $\int_0^\infty \pi(dr)\, re^{-r\lambda}=D-\Psi'(\lambda)$. To derive the third one from the second one,   
we use the identity 
$\partial_\lambda  u \,  (t, \lambda)=- \Ppsi(\lambda)^{-1} \partial_t u(t, \lambda)$ and we do an integration by part. Recall $B(\lambda_1,\lambda_2)$ from (\ref{Blamb}). By the previous computations we get 
$$\tilde{B}(\lambda_1,\lambda_2)= x^2B(\lambda_1,\lambda_2) - x^2\partial_{\lambda}u(s,\lambda_2)\partial_{\lambda}u(t,\lambda_1+u(s,\lambda_2))e^{-Dt}e^{-xu(t,\lambda_1+u(s,\lambda_2))}$$
Recall that we already proved that 
$x^2\int_0^\infty\!\!\int_0^\infty d\lambda_1 d\lambda_2\big(B(\lambda_1,\lambda_2)-B(\lambda_1,\lambda_2+\theta)\big)$ equals the right member of (\ref{kieq}). So, to complete 
the proof, we set 
$$F(\theta):=\int_0^\infty\!\!\!\int_0^\infty\!\!\! d\lambda_1 \, d\lambda_2 \, \partial_{\lambda}u(s,\lambda_2+\theta) \, \partial_{\lambda}u(t,\lambda_1+u(s,\lambda_2+\theta)) \, e^{-Dt}e^{-xu(t,\lambda_1+u(s,\lambda_2+\theta))}$$
and calculations similar as in the infinite variation case yield $x^2(F(0)-F(\theta))=-A_2$, which entails the desired result in the finite variation cases. 
\cqfd 
\end{proof}
To complete the proof of Theorem \ref{mainth}, we need the following technical lemma whose proof is postponed. 
\begin{lemma}
\label{pntconv} We assume that $\Ppsi$ is not linear. 
Then, $\bP$-a.s.~for all $y\in [0, x]$, $\lim_{t\rightarrow \infty} M_t (\{y\})$ exists. 
\end{lemma}

\subsubsection{Proof of Theorem \ref{mainth} (\textit{ii-a}).} 
We temporarily admit Lemma \ref{pntconv}. We assume that $\gamma \!>\!0$ and that $\Ppsi$ is conservative. To simplify notation, we denote by $\{ Z\! \rightarrow \! \infty\}$ the event $\{  \lim_{t\rightarrow \infty} Z_t \! =\! \infty \}$. Recall from (\ref{defP}) and  (\ref{defQ}) the definition of the Poisson point measures $\ccP$ and $\ccQ$. 
For any $t \! \in\! (0, \infty)$, we define the following: 
\begin{equation} \label{defPQt}
\ccP_t = \sum_{i\in I}   \delta_{(x_i, \rZ^i_{\cdot \wedge t})} 
\qquad \textrm{and} \qquad
 \ccQ_t= \sum_{j\in J}   \un_{\{ t_j \leq t \} }\delta_{(x_j, t_j,  \rZ^j_{\cdot \wedge (t-t_j)}) }\ .
 \end{equation}
We then define $\ccG_t$ as the sigma-field generated either by 
$\ccP_t$ if $\Ppsi$ is of infinite variation type, or by $\ccQ_t$ if $\Ppsi$ is of finite variation type. The Markov property (see Lemma \ref{LemmaMarkov}) applied to the process $(Z_t,t\geq 0)$ in the filtration $\ccG_t,t\geq 0$ yields
\begin{equation*}
\bP(R^{-1}_t (U) \neq  R_{t+s}^{-1} (V)  \, ; Z \! \rightarrow \! \infty  ) = 
\bE \big[  \un_{\{ R^{-1}_t (U) \neq  R_{t+s}^{-1} (V) \} } \big(1\!-\! e^{-\gamma Z_{t+s}} \big) \big]
\end{equation*}
By Lemma \ref{keyL} and the identity $u(s,\gamma)=\gamma$, we then get  
\begin{equation*}  \bP(R^{-1}_t (U) \neq  R_{t+s}^{-1} (V)  \, ; Z \! \rightarrow \! \infty  ) =
  x^2 \! \! \int_0^{v(t)} \!\! \!\!\! dw \,  \Ppsi (w)  e^{-xw} \, 
\frac{u(\!-t, w)\!-\! \big( u(\!-t, w) \!-\!\gamma \big)_+}{\Ppsi (u(\!-t, w))} =: A(t)  
\end{equation*}
We set $\ree \! = \! R^{-1}_\infty (V)$. Using the Portmanteau theorem as $s \! \rightarrow \!\infty$ on the law of the pair $(R^{-1}_t (U) , R_{t+s}^{-1} (V))$ with the complement of the closed set $\{(y,y):y\in[0,x]\}$, we get $ \bP(R^{-1}_t (U) \! \neq \! \ree  ;Z \! \rightarrow \! \infty  ) \!\leq \! A(t)$. But now observe that 
$\bE[ \un_{\{ R^{-1}_t (U) \neq  \ree   ; Z  \rightarrow \infty \}} \, | \, \ccP , V  ]= 
(1-M_t (\{ \ree\} ))\un_{\{  Z \rightarrow \infty  \}}$. Thus , 
\begin{equation} \label{probeve} 
\bE\big[ (1-M_t (\{ \ree\} ))\un_{\{ Z \rightarrow \infty  \}} \big] \leq A(t)
\end{equation}
We next prove that $\lim_{t\rightarrow \infty} A(t)= 0$. 
First note that for all $w \! \in \! (0, \gamma )$, $w \!<\! v(t)$ and  
$u(-t, w)\! < \! \gamma$, moreover $u(-t, w) \! \downarrow \!0$ 
as $t \! \uparrow \! \infty$. 
Since $\Ppsi^\prime (0+) \!= \! -\infty$, $\lambda / \Ppsi (\lambda) \! \uparrow \! 0$ as 
$  \lambda  \! \downarrow \! 0$. This implies that 
$$  x^2 \int_0^\gamma \!\!\! dw \,  \Ppsi (w) \, e^{-xw} \, 
 \frac{u(\!-t, w)\!-\! \big( u(\!-t, w) \!-\! \gamma \big)_+}{\Ppsi (u(\!-t, w))} = x^2 \int_0^\gamma
  \!\!\! dw \,  \Ppsi (w) \, e^{-xw} \, 
 \frac{u(\!-t, w)}{\Ppsi (u(\!-t, w))} 
 \underset{^{t\rightarrow \infty}}{-\!\!\!-\!\!\! \longrightarrow } 0 . $$
If $\gamma \! = \! \infty$, then, this proves $\lim_{t\rightarrow \infty} A(t)= 0$. Let us assume that 
$\gamma \! < \! \infty$: for all $w \! \in \!  (\gamma, v(t))$, $u(-t, w) \!>\! \gamma$ and we get 
\begin{equation} \label{apresg}
 x^2 \int_\gamma^{v(t)} \!\!\! dw \,  \Ppsi (w) \, e^{-xw} \, 
 \frac{u(\!-t, w)\!-\! \big( u(\!-t, w) \!-\! \gamma \big)_+}{\Ppsi (u(\!-t, w))}= x^2 \int_\gamma^{v(t)} \!\!\! dw \,  \Ppsi (w) \, e^{-xw} \, 
 \frac{\gamma}{\Ppsi (u(\!-t, w))} \; .
  \end{equation}
There are two cases to consider: 
if $\Ppsi$ is persistent, then $v(t) \! =\!  \infty$. Moreover, for all $w\! \in\!  (\gamma, \infty)$, 
$u(-t, w)$ is well-defined and $u(-t, w)\! \uparrow \! \infty$ as $t\! \uparrow \! \infty$, which implies that 
(\ref{apresg}) tends to $0$ as $t\! \rightarrow \! \infty$. If $\Ppsi$ is non-persistent, then 
$v(t)\! < \! \infty$. Observe that $\lim_{t\rightarrow \infty} v(t)\! =\! \gamma$
and use (\ref{ratio}) with $\lambda\! =\! u(-t, w)$ to prove that $w<u(-t,w)$ for any 
$w \!\in \! ( \gamma , v(t)) $. Since $\Ppsi$ increases, we get 
$$  x^2 \int_\gamma^{v(t)} \!\!\! dw \,  \Ppsi (w) \, e^{-xw} \, 
 \frac{\gamma}{\Ppsi (u(\!-t, w))} \leq \gamma x^2 \int_\gamma^{v(t)} \!\!\! dw \,  e^{-xw}  \underset{^{t\rightarrow \infty}}{-\!\!\!-\!\!\! \longrightarrow } 0 . $$
This completes the proof of  $\lim_{t\rightarrow \infty} A(t)= 0$.

By (\ref{probeve}) and Lemma \ref{pntconv}, 
we get $\bP$-a.s.~on $\{  Z \! \rightarrow \! \infty\}$, $M_t(\{ \ree\})\! \rightarrow \!1$. 
Thus, it entails $\lVert M_t \!-\!\delta_\ree \rVert_{\textrm{var}}\! \rightarrow \! 0$ by Lemma \ref{detcvvar}, as $t\! \rightarrow \!\infty$, which implies Theorem \ref{mainth} (\textit{ii-a}).  \cqfd

\subsubsection{Proof of Theorem \ref{mainth} (\textit{iii-a}).} We assume that $\Ppsi$ is persistent, of infinite variation type and such that $\gamma < \infty$. Observe that $\ccP$ under $\bP (\, \cdot \, | \, \lim_{t\rightarrow \infty} Z_t = 0)$ is a Poisson point measure associated with the branching mechanism $\Ppsi (\cdot + \gamma)$ that is sub-critical (and therefore conservative). So the proof of 
Theorem \ref{mainth} (\textit{iii-a}) reduces to the cases of sub-critical persistent branching mechanisms and without loss of generality, we now assume that $\Ppsi$ is so. Thus, 
$\lim_{\theta \rightarrow \infty} u(t, \theta) \!=\! v(t)\!=  \! \infty$. By letting $\theta$ go to $\infty$ in 
Lemma \ref{keyL}, we get 
$$  \bP(R^{-1}_t (U) \neq  R_{t+s}^{-1} (V) )=  x^2 \int_0^\infty 
 \frac{u(-t, w)}{\Ppsi (u(-t, w))} \, \Ppsi (w) \, e^{-xw} \,  dw =: B (t) \; , $$
which does not depend on $s$. Then, set $\ree \! = \! R^{-1}_\infty (V)$. 
By the Portmanteau theorem as $s \! \rightarrow \!\infty$, we get 
$ \bP(R^{-1}_t (U) \! \neq \! \ree   ) \!\leq \! B(t)$. Next observe that $\bE[ \un_{\{ R^{-1}_t (U) \neq  \ree    \}} \, | \, \ccP , V  ]= 
1-M_t (\{ \ree\} ) $. Therefore, 
\begin{equation} \label{reprobeve} 
0 \leq 1-\bE\big[ M_t (\{ \ree\} ) \big] \leq B(t)
\end{equation}
Since $\Ppsi$ is sub-critical and persistent 
for all $w\!\in \!(0, \infty)$, $u(-t, w)$ increases to $\infty$ as $t \! \uparrow \! \infty$.
Moreover, since $\Ppsi$ is of infinite variation type, $\lambda/ \Ppsi (\lambda)$ decreases to $0$ as $\lambda \uparrow \infty$, which implies that  $\lim_{t\rightarrow \infty} B (t) = 0$. By (\ref{reprobeve}) and Lemma \ref{pntconv}, 
we get $\bP$-a.s.~$M_t(\{ \ree\})\! \rightarrow \!1$, and thus $\lVert M_t \!-\!\delta_\ree \rVert_{\textrm{var}}\! \rightarrow \! 0$ by Lemma \ref{detcvvar}, as $t\! \rightarrow \!\infty$, which completes the proof of Theorem \ref{mainth} (\textit{iii-a}). \cqfd

\subsubsection{Proof of Lemma \ref{pntconv}.}
\label{pntconvpf}
 To complete the proof of Theorem \ref{mainth}, it only remains to prove Lemma \ref{pntconv}. We shall proceed by approximation, in several steps. Recall the filtration $\ccG_t,t\geq 0$ introduced below (\ref{defPQt}).

\begin{lemma}\label{LemmaMarkov}
\label{pntMrkv} Assume that $\Ppsi$ is conservative and not linear. 
Then, for all $s, t, \lambda \in [0, \infty)$ 
$$ \textrm{$\bP$-a.s.} \qquad \bE \big[  \, e^{-\lambda Z_{t+s}} \, \big| \, \ccG_t  \, \big]  
= e^{-u(s, \lambda) Z_t } . 
$$
 \end{lemma}
\begin{proof}  
We first consider the infinite variation cases. We fix 
$s_0, \varepsilon \in (0, \infty)$. For any $t\in (s_0, \infty)$, we set 
\begin{equation}
\label{Pee}
\ccP^{>\ee}_t \!= \! \sum_{i\in I} \un_{\{ \rZ^i_{s_0} \! > \ee \}} 
\delta_{(x_i \, , \,  \rZ^i_{\, \cdot \, \wedge t })}
\quad \textrm{and} \quad  
\cZ^\ee_t \!= \! \sum_{i\in I} \un_{\{ \rZ^i_{s_0} \! > \ee \}} \rZ^i_{t} \delta_{x_i} .  
\end{equation}
Since $t\! >\! s_0$, and by monotone convergence for sums, $\lim_{ \varepsilon \rightarrow 0} \cZ^{ \varepsilon}_{t+s} ([0, x]) \!= \! Z_{t+s}$. 
Then, observe 
that $\cZ_{t+s}^{\ee}$ is independent from $\ccP_t \!- \!  \ccP^{>\ee}_t$. 
Thus, $\bP$-a.s.~$ \bE [   e^{-\lambda Z_{t+s}}  |  \ccG_t  ]  
=\lim_{ \varepsilon \rightarrow 0} \bE [   e^{-\lambda \cZ^\ee_{t+s} ([0, x])}  |  \ccP^{>\ee}_t  ]$.

  Next, note that $\ccP^{>\ee}_t $ is a Poisson point measure whose law is specified as follows. 
By Theorem \ref{cluster} (b) and Lemma \ref{nufi}, first note that $\rN_\Ppsi (\rZ_{s_0} \!> \! \ee)\! = \! \nu_{s_0} ((\ee, \infty]) \! \in\!  (0, \infty)$. Then, 
$Q_{s_0, \ee} \!= \!  \rN_\Ppsi (\, \cdot \, | \, \rZ_{s_0} \!> \! \ee)$ is 
a well-defined probability on $\bbD ([0, \infty), [0, \infty])$. 
Theorem \ref{cluster} (c) easily entails that 
\begin{equation}    
\label{redu2}
\textrm{$Q_{s_0, \ee} $-a.s.}\qquad Q_{s_0, \ee} \big[ \, e^{-\lambda \rZ_{t+s}} \, \big| \, 
\rZ_{\, \cdot \, \wedge t} \,  \big] = e^{-u(s, \lambda) \rZ_t }  \; .
\end{equation}
Next, note that $\ccP^{>\ee}_t$ can be written as $\sum_{1\leq k \leq S} \delta_{(X_k, Y^k_{\, \cdot \, \wedge t})}$, 
where $(X_k, Y^k )$, $k\!\geq \!1$, is an i.i.d.~sequence of 
$[0, x] \! \times \!  \bbD([0, \infty), [0, \infty])$-valued r.v.~whose law is 
$x^{-1}\un_{[0, x]}(y)  \ell (dy) \, Q_{s_0, \ee} (d\rZ)$ and where $S$ is a Poisson r.v.~with mean $x\nu_{s_0} ((\ee, \infty]) $ that is independent from 
the $(X_k, Y^k)_{k\geq 1}$. By an easy argument, we derive from (\ref{redu2}) that 
$\bP$-a.s.~$\bE [  e^{-\lambda \cZ^\ee_{t+s} ([0, x])} |  \ccP^{>\ee}_t ] = 
 e^{-u(s, \lambda) \cZ^\ee_{t} ([0, x])}$, which entails the desired result as $ \varepsilon \rightarrow 0$. 

\medskip

In the finite variation cases,  
we also proceed by approximation: for any $ \varepsilon \in (0, \infty)$, we set 
$$\ccQ^{>\ee}_t \!= \! \sum_{j\in J} \un_{\{ t_j \leq t \, ; \,   \rZ^j_{0} \! > \ee \}} 
\delta_{(x_j \, , \,  t_j \, , \, \rZ^j_{\, \cdot \, \wedge (t-t_j) })} , \quad 
\cZ^\ee_t \!= \! \sum_{j\in J} \un_{\{  t_j \leq t \, ; \,  \rZ^j_{0} \! > \ee \} } \rZ^j_{t-t_j} \delta_{x_j} \quad \textrm{and} \quad Z^*_t= Z_t -xe^{-Dt} . $$
Then, note that 
$\lim_{ \varepsilon \rightarrow 0} \cZ^{ \varepsilon}_{t+s} ([0, x]) \!= \! Z^*_{t+s}$ and  
observe that $\cZ_{t+s}^{\ee}$ is independent from $\ccQ_t \!- \!  \ccQ^{>\ee}_t$. 
Thus, $\bP$-a.s.~$ \bE [   e^{-\lambda Z^*_{t+s}}  |  \ccG_t  ]  
=\lim_{ \varepsilon \rightarrow 0} \bE [   e^{-\lambda \cZ^\ee_{t+s} ([0, x])}  |  \ccQ^{>\ee}_t  ]$.  
Next, note that $\ccQ^{>\ee}_t $ is a Poisson point measure that can be written as 
$\sum_{1\leq k \leq S} \delta_{(X_k, T_k, Y^k_{\, \cdot \, \wedge (t-T_k)})}$ where 
$(X_k, T_k, Y^k )_{k \geq 1}$, is an i.i.d.~sequence of 
$[0, x] \! \times \! [0,t] \! \times \!  \bbD([0, \infty), [0, \infty])$-valued r.v.~whose law is 
$x^{-1}\un_{[0, x]}(y)  \ell (dy) \, (1-e^{-Dt})^{-1}De^{-Ds} \ell (ds) \, Q_{\ee} (d\rZ)$ where
$$ Q_{\ee}(d\rZ ) :=  \frac{_1}{^{\pi((\varepsilon,\infty))}} 
\int_{(\varepsilon,\infty)} \!\!\!\!\!\!\!  \pi(dr) \, \bbP_r( d\rZ) \; $$
and $S$ is an independent Poisson r.v.~with mean 
$x(1\!-\!e^{-Dt})D^{-1}\pi ((\ee, \infty)) $. When $D\!=\!0$, one should replace $(1\!-\!e^{-Dt})D^{-1}$ by $t$ in the last two expressions. Note that the Markov property applies under $Q_\ee$. Namely, $Q_\ee$-a.s.~$ Q_{\ee}[  e^{-\lambda \rZ_{t+s}}  |  
\rZ_{\, \cdot \, \wedge t} ] = e^{-u(s, \lambda) \rZ_t }$. This implies $\bP$-a.s.~the following
\begin{equation} \label{soust}
\bE \Big[  \, e^{-\lambda \sum_{j\in J} \un_{\{t_j \leq t\, , \, \rZ^j_{0} \! > \ee \}} \rZ^j_{t+s-t_j}} \, \big| \, 
\ccQ^{>\ee}_t  \, \Big]  =  
 e^{-u(s, \lambda) \sum_{j\in J} \un_{\{t_j \leq t\, , \, \rZ^j_{0} \! > \ee \}} \rZ^j_{t-t_j}}  =   
 e^{-u(s, \lambda) \cZ^\ee_{t} ([0, x])} .
 \end{equation}
Then, note that $\sum_{j\in J} \un_{\{t < t_j \leq t+s\, , \, \rZ^j_{0} \! > \ee \}} \rZ^j_{t+s-t_j}$ is 
independent from $\ccQ^{>\ee}_t$. By the exponential formula for Poisson point measures, we thus $\bP$-a.s.~get 
\begin{equation} \label{surt} 
\! - \! \log \bE \Big[   e^{-\lambda \sum_{j\in J} \un_{\{t < t_j \leq t+s\, , \, \rZ^j_{0} \! > \ee \}} \rZ^j_{t+s-t_j}} \, 
\big| \, \ccQ^{>\ee}_t  \, \Big]  =   xe^{-Dt}\!\!\int_0^{s}\!\!\! da \, e^{-Da}\!\!\int_{(\varepsilon,\infty)}\!\!\!\! 
\!\!\!\!\pi(dr)\big(1 \!- \! e^{-r u(s-a,\lambda)}\big) .
\end{equation}
As $ \varepsilon \! \rightarrow \!0$, 
the right member of (\ref{surt}) tends to $xe^{-Dt}\!\!\int_0^{s} da \, e^{-Da} \big( D u(s\!-\!a, \lambda) 
\! - \! \Ppsi (u(s\!-\!a, \lambda))\big)$ that is equal to $xe^{-Dt}u(s, \lambda)-x\lambda e^{-D(s+t)}$ by a simple integration by parts. This computation combined with (\ref{soust}) and (\ref{surt}), implies 
$$ \lim_{ \varepsilon \rightarrow 0}-\log \bE \big[ e^{-\lambda \cZ^{ \varepsilon}_{t+s} ([0, x]) } \big| \ccQ^{ >\varepsilon }_t \big]= u(s, \lambda) \lim_{ \varepsilon \rightarrow 0} \cZ^{ \varepsilon}_t ([0, x]) + 
xe^{-Dt}u(s, \lambda)-x\lambda e^{-D(s+t)} $$
Namely, $-\log \bE [ e^{-\lambda Z^*_{t+s}} | \ccG_t]= u(s, \lambda) ( Z^*_t + xe^{-Dt})  
- \lambda x e^{-D(s+t)}=  u(s, \lambda) Z_t - \lambda x e^{-D(s+t)}$, which implies the desired result.  \cqfd 
\end{proof}

\begin{lemma}
\label{eemg} We assume that $\Ppsi$ is conservative and non-linear. 
We fix $s_0, \varepsilon \!\in \!(0, \infty)$. For any $t\! \in \!(s_0, \infty)$, we define 
$\cZ^{ \varepsilon}_t$ as follows:

-- If $\Ppsi$ is of infinite variation type, then 
$\cZ^{ \varepsilon}_t= \sum_{i\in I} \un_{\{ \rZ^i_{s_0} > \varepsilon \}} \rZ^i_t \delta_{x_i}$.

--  If $\Ppsi$ is of finite variation type, then $\cZ^{ \varepsilon}_t= \sum_{j\in J} \un_{\{ t_j \leq s_0 \, ; \,  \rZ^j_0 > \varepsilon \}} \rZ^j_{t-t_j} \delta_{x_j}$. 

\noi
Recall the definition of the sigma-field $\ccG_t$.   
Then, for all $t \!\in \! (s_0, \infty)$, all $s, \theta \!\in \! [0, \infty)$ and all $y \!\in\! [0, x]$, 
\begin{equation}
\label{eemgg}
 \textrm{$\bP$-a.s}\qquad \bE \big[  \un_{\{ Z_{t+s} >0\}}  
\, \frac{_{\cZ^{\ee}_{t+s} ([0, y])}}{^{Z_{t+s}}} \, \big( 1-e^{-\theta Z_{t+s}}\big) \, 
\big| 
\, \ccG_t \,  \big] = \un_{\{ Z_{t} >0\}}   \frac{_{\cZ^{\ee}_{t} ([0, y])}}{^{Z_{t}}}  \, \big( 1-e^{-u(s, \theta) Z_{t}}\big) .
\end{equation}
\end{lemma}
\begin{proof} We first consider the infinite variation cases. Note that in these cases,  
$\cZ^{ \varepsilon}_t$ is defined as in (\ref{Pee}). Let $\lambda \! \in \! (0, \infty)$. Recall the notation $Q_{s_0, \ee}\! =\!  \rN_\Ppsi (\, \cdot \, | \rZ_{s_0} \! > \!  \varepsilon)$ from the proof of Lemma \ref{pntMrkv}: by differentiating (\ref{redu2}), we get
\begin{equation}    
\label{Qder}
\textrm{$Q_{s_0, \ee} $-a.s}\qquad Q_{s_0, \ee} \big[ \,\rZ_{t+s} e^{-\lambda \rZ_{t+s}} \, \big| \, 
\rZ_{\, \cdot \, \wedge t} \,  \big] =
\rZ_t \, e^{-u(s, \lambda) \rZ_t } \, \partial_\lambda  u \, (s, \lambda) . 
\end{equation}
Recall from (\ref{defPQt}), the definition of $\ccP_t$. 
Let $F$ be a bounded nonnegative measurable function on 
the space of point measures on $[0, x] \! \times \! \bbD([0, \infty), [0, \infty ])$. We then set 
$A (\lambda)= \bE [\cZ^\ee_{t+s}([0, y]) \, e^{-\lambda Z_{t+s}} F(\ccP_t) ]$.  
By Palm formula (\ref{Palm}), Lemma \ref{pntMrkv} and (\ref{Qder}), we get 
\begin{eqnarray*}
A(\lambda) \!\!  \!\! &= &\!\! \!\! \bE \Big[ \sum_{^{i\in I}} \un_{\{ x_i \in [0, y] \, ; \, \rZ^i_{s_0} > \ee \} } \rZ^i_{t+s} \, 
e^{-\lambda \rZ^{i}_{t+s}}  \; e^{-\lambda \sum_{k\in I \backslash \{ i\}} \rZ^k_{t+s}  } \; F \big(\,  \delta_{(x_i \, , \, \rZ^i_{\, \cdot \, \wedge t} )} + \ccP_t \!-\! \delta_{(x_i \, , \,  \rZ^i_{\, \cdot \, \wedge t})}\,  \big) \,  \Big]  \\
 \!\!  \!\! &= &\!\! \!\! \nu_{s_0} ((\ee, \infty]) \!\!  \int_0^y \! \!\! dr \!\! \int \! Q_{s_0, \ee} (d\rZ) \,
\bE \Big[  \rZ_{t+s} 
e^{-\lambda \rZ_{t+s}} \, e^{-\lambda Z_{t+s}} F \big( \, \delta_{(r \, , \,  \rZ_{\, \cdot \, \wedge t} )} \!+\!  \ccP_t \,  \big) \,  \Big] \\
 \!\!  \!\! &= &\!\! \!\! 
\big(\partial_\lambda  u(s, \lambda) \big) 
 \, \nu_{s_0} ((\ee, \infty]) \!\!  \int_0^y \! \!\! dr \!\! \int \!\! Q_{s_0, \ee} (d\rZ) \,
\bE \Big[ \rZ_{t} 
e^{-u(s,\lambda) \rZ_{t}} \, e^{-u(s,\lambda) Z_{t}} F \big( \, \delta_{(y \, , \,  \rZ_{\, \cdot \, \wedge t} )} \!\!+\!\!  \ccP_t \,  \big) \,  \Big] \\
 \!\!  \!\! &= &\!\! \!\! 
\big(\partial_\lambda  u(s, \lambda) \big) 
\, \bE \big[ \cZ^\ee_{t}([0, y]) e^{-u(s, \lambda) Z_{t}} F(\ccP_t) \big] .          
\end{eqnarray*}
By an easy argument, it implies that $\bP$-a.s.~for all $\lambda \!\in\! (0, \infty)$,
$$  \bE \big[ \, \cZ^\ee_{t+s}([0, y]) \, e^{-\lambda Z_{t+s}} \, \big| \,\ccG_t \,  \big] = \cZ^\ee_{t}([0, y]) \, e^{-u(s, \lambda) Z_{t}} \, \partial_\lambda  u(s, \lambda)  . $$
Thus, $\bP$-a.s.~for all $ \lambda , \theta \!\in \! (0, \infty)$, 
\begin{eqnarray}
\label{subt}
 \bE \big[ \un_{\{ Z_{t+s} >0\}} \cZ^\ee_{t+s}([0, y]) e^{-\lambda Z_{t+s}} \big( 1-e^{-\theta Z_{t+s}} \big)
 \, \big| \,\ccP_t \,  \big] & = & \nonumber \\ 
   & & \hspace{-65mm} \un_{\{ Z_{t} >0\}} \cZ^\ee_{t}([0, y]) \, \big( e^{-u(s, \lambda) Z_{t}}  \partial_\lambda  u(s, \lambda) - e^{-u(s, \lambda+ \theta) Z_{t}}  \partial_\lambda  u(s, \lambda\!+\!\theta) \big).
\end{eqnarray}  
When we integrate the first member of (\ref{subt}) in $\lambda$ on $(0, \infty)$, we get the first member of 
(\ref{eemgg}). Then, by an easy change of variable, we get 
\begin{equation}
\label{chgvar}
\forall \lambda_0 \in [0, \infty), \; \forall z \in (0, \infty), \quad  
\int_{\lambda_0}^{\infty} \!\!\! d\lambda \, e^{-u(s, \lambda) z}  \, \partial_\lambda  u(s, \lambda) = \frac{_1}{^z} \big(e^{-u(s, \lambda_0) z}  \! -\! e^{v(s) z} \big), 
\end{equation}
where we recall that $v(s)\!= \! \lim_{\lambda \rightarrow \infty} u(s, \lambda)$, which is infinite if $\Ppsi$ is persistent and finite otherwise. 
Since $\Ppsi$ is conservative, recall that $\kappa (s)\!= \! \lim_{\lambda \rightarrow 0+} u(s, \lambda)= 0$. 
Thus, when we integrate the second member of (\ref{subt}) in $\lambda$ on $(0, \infty)$, we obtain the second member of (\ref{eemgg}), which completes the proof of the lemma in the infinite variation cases.

\medskip

We next consider the finite variation cases. Note that the definition of $\cZ^{ \varepsilon}$ is slightly different from the proof of Lemma \ref{pntMrkv}. Recall from (\ref{defPQt}), the definition of $\ccQ_t$. 
Let $F$ be a bounded nonnegative measurable function on 
the space of point measures on $[0, x] \! \times  \! [0, \infty) \! \times \! \bbD([0, \infty), [0, \infty ])$. We set 
$A (\lambda)= \bE [\cZ^\ee_{t+s}([0, y]) \, e^{-\lambda Z_{t+s}} F(\ccQ_t) ]$.  
By Palm formula (\ref{Palm}) and Lemma \ref{pntMrkv} we get 
\begin{eqnarray*}
A(\lambda) \!\!  \!\!\! &= &\!\! \!\!\! \bE \Big[ \sum_{^{j\in J}} 
\un_{\{ x_j \in [0, y] \, ; \, t_j\leq s_0 \, ; \, \rZ^j_{0} > \ee \} } \rZ^j_{t+s-t_j} \, 
e^{-\lambda \rZ^{j}_{t+s-t_j}}  \; e^{-\lambda \sum_{k\in J \backslash \{ j\}}\un_{\{ t_k\leq t+s\}} \rZ^k_{t+s-t_k}  } \; e^{-\lambda xe^{-D(t+s)}} \;\\
&&\hspace{27pt}\times F \big(\,  \delta_{(x_j \, , \, t_j \, , \, \rZ^j_{\, \cdot \, \wedge (t-t_j)} )} + \ccQ_t \!-\! \delta_{(x_j \, , \, t_j \, , \, \rZ^j_{\, \cdot \, \wedge (t-t_j)} )}\,  \big) \,  \Big]  \\
\!\!  \!\!\! &= &\!\! \!\!\!  \int_0^y \! \!\! da \!\! \int_0^{s_0} \! \!\! db \, e^{-Db} \!\! \int_{(\ee,\infty)} \! \pi(dr) \,
\bbE_r\bigg[ \bE\Big[\rZ_{t+s-b}e^{-\lambda \rZ_{t+s-b}}e^{-\lambda Z_{t+s}} F \big( \, \delta_{(a \, , \, b \, , \,  \rZ_{\, \cdot \, \wedge (t-b)} )} \!+\!  \ccQ_t \,  \big) \,  \Big]\bigg] \\
\!\!  \!\!\! &= &\!\! \!\!\!  \partial_\lambda  u(s, \lambda)\!\!\! \int_0^y \! \!\!\! da \!\! \int_0^{s_0} \! \!\!\!\! db \, e^{-Db} \!\! \int_{(\ee,\infty)} \!\!\!\!\!\! \pi(dr) \,
\bbE_r\bigg[ \bE\Big[\rZ_{t-b}e^{-u(s, \lambda) \rZ_{t-b}}e^{-u(s, \lambda) Z_{t}} F \big( \, \delta_{(a \, , \, b \, , \,  \rZ_{\, \cdot \, \wedge (t-b)} )} \!+\!  \ccQ_t \,  \big) \,  \Big]\bigg] \\
\!\!  \!\!\! &= &\!\! \!\!\!  \partial_\lambda  u(s, \lambda) \bE [\cZ^\ee_{t}([0, y]) \, e^{-u(s, \lambda) Z_{t}} F(\ccQ_t) ]. 
\end{eqnarray*}
Then, we argue exactly as in the infinite variation cases. \cqfd 
\end{proof}

We now complete the proof of Lemma \ref{pntconv}. If 
$\Ppsi$ is not conservative, then we have already proved that on $\{ \zeta_\infty  \! < \! \infty\}$, $M$ has an Eve in finite time. Moreover, conditionally on $\{ \lim_{t\rightarrow \infty} Z_t \!= \! 0 \}$, $M$ is distributed as the frequency process of a CSBP($\Ppsi (\cdot + \gamma)$) that is sub-critical, and therefore conservative. Thus, without loss of generality, we assume that $\Ppsi$ is conservative. In 
this case, Lemma \ref{eemg} applies: 
we fix $s_0, \varepsilon \! \in \!  (0, \infty)$ and we let 
$\theta$ go to $\infty$ in (\ref{eemgg}); this implies that $t \mapsto  \un_{\{ Z_{t} >0\}}   \frac{_{\cZ^{\ee}_{t} ([0, y])}}{^{Z_{t}}}$ is a super-martingale. Then,  
$$ \textrm{$\bP$-a.s.} \quad \forall q\in \bbQ \cap [0, x] \, , \quad \lim_{t\rightarrow \infty}  \un_{\{ Z_{t} >0\}}   \frac{_{\cZ^{\ee}_{t} ([0, q])}}{^{Z_{t}}}=: R^{ \varepsilon}_q \; \textrm{exists}. $$ 
Then observe there exists a finite subset $S_{s_0, \varepsilon}\!:= \! \{ X_1\!< \! \ldots \! <\!  X_N\} \! \subset\! [0, x]$ 
such that a.s.~for all $t \! \in \! (s_0, \infty)$, 
$\cZ^{ \varepsilon}_t ([0, x] \backslash S_{s_0, \varepsilon})\! =\! 0$. Then, for any $1\! \leq \! k \! \leq \! N$, 
there exists $q, q^\prime \in \bbQ \! \cap \! [0, x]$ such that $q\! < \! X_k \! < \! q^\prime$ and $\un_{\{ Z_t >0\}} M_{t} (\{ X_k\}) \!= \! \un_{\{ Z_{t} >0\}}  \cZ^{\ee}_{t} ((q, q^\prime])/ Z_{t} \! \longrightarrow \!  R^{ \varepsilon}_{q^\prime} \!-\! R^{ \varepsilon}_q$, as 
$t\! \rightarrow \!\infty$. 

Now observe that if $\Ppsi$ is of infinite variation type, 
$\{ x_i \, ; \, i\! \in \! I\}= \bigcup_{n,m \in \bbN} S_{2^{-m}, 2^{-n}}$. Thus, on the event 
$\{ \zeta_0 \!=\! \infty\}$ (no extinction in finite time), this entails that $\bP$-a.s.~for all $i\! \in \! I $, 
$\lim_{t\rightarrow \infty} M_t (\{ x_i\})$ exists. Moreover, for all $y\! \notin\!  \{ x_i \, ; \, i\! \in \! I\}$ and all $t\! \in \!(0, \infty)$, $M_t (\{ y\})= 0$. Finally, on $\{ \zeta_0 \! < \! \infty\}$, we have already proved that 
$M$ has an Eve in finite time. This completes the proof of Lemma \ref{pntconv} when $\Ppsi$ is of infinite variation type.     

If $\Ppsi$ is finite variation type, note that $\{ x_j \, ; \, j\! \in \! J\}= \bigcup_{n,m \in \bbN} S_{m, 2^{-n}}$. Since there is no extinction in finite time, we get that $\bP$-a.s.~for all $j\! \in \! J $, 
$\lim_{t\rightarrow \infty} M_t (\{ x_j\})$ exists, which completes the proof since for all 
$y\! \notin\!  \{ x_j \, ; \, j\! \in \! J\}$ and all $t\! \in \! (0, \infty)$, we have $M_t (\{ y\})= 0$. \cqfd

%

\end{document}